\documentclass{article}
\usepackage{graphicx} 

\usepackage{wrapfig}
\usepackage{xspace}
\usepackage{relsize}

\newcommand{\lv}{\left \lvert}
\newcommand{\rv}{\right \rvert}

\newcommand{\simiid}{\overset{\iid}{\sim}}
\newcommand{\holder}{H\"older\xspace}
\newcommand{\unorm}{\|\cdot\|_{\infty}}

\usepackage{tikz}
\usetikzlibrary{arrows.meta,positioning,calc,decorations.pathreplacing, patterns}
\usepackage{pgfplots}
\pgfplotsset{compat=1.17}

\title{Optimal Anytime-Valid Tests for Composite Nulls}
\author{Shubhanshu Shekhar \\ EECS Department \\ University of Michigan, Ann Arbor \\ \url{shubhan@umich.edu}}
\date{}

\usepackage{amsmath,amssymb,amsthm,amsfonts,latexsym,bbm,xspace,graphicx,float,mathtools,mathdots}
\usepackage{braket,caption,ellipsis,xcolor,textcomp}
\usepackage{combelow} %
\usepackage{booktabs}
\usepackage{hyperref}
\usepackage{algorithm}
\usepackage{algpseudocode}

\usepackage[letterpaper,margin=1in]{geometry}
\usepackage{enumitem}

\newtheorem{theorem}{Theorem}

\usepackage{chngcntr}
\counterwithin{theorem}{section}

\newtheorem*{namedtheorem}{\theoremname}
\newcommand{\theoremname}{testing}

\newtheorem{lemma}{Lemma}

\newtheorem{proposition}{Proposition}
\newtheorem{fact}{Fact}

\newtheorem{assumption}{Assumption}

\theoremstyle{definition}
\newtheorem{definition}{Definition}
\newtheorem{remark}{Remark}

\newcommand{\R}{\mathbbm R}

\newcommand{\calB}{\mathcal{B}}
\newcommand{\calC}{\mathcal{C}}

\newcommand{\calF}{\mathcal{F}}

\newcommand{\calH}{\mathcal{H}}

\newcommand{\calJ}{\mathcal{J}}

\newcommand{\calM}{\mathcal{M}}
\newcommand{\calN}{\mathcal{N}}
\newcommand{\calO}{\mathcal{O}}
\newcommand{\calP}{\mathcal{P}}

\newcommand{\calX}{\mathcal{X}}
\newcommand{\calY}{\mathcal{Y}}

\newcommand{\ignore}[1]{}

\usepackage{autonum}

\usepackage{natbib}
\setcitestyle{authoryear,open={(},close={)}} %

\newcommand{\EE}{\mathbb{E}}
\newcommand{\PP}{\mathbb{P}}
\DeclareMathOperator*{\argmin}{argmin}
\DeclareMathOperator*{\argmax}{argmax}

\usepackage{subcaption}
\usepackage[capitalize,noabbrev]{cleveref}
\crefname{ineq}{inequality}{inequalities}
\crefname{ineq}{inequality}{inequalities}
\crefname{assumption}{Assumption}{Assumptions}
\crefname{proposition}{Proposition}{Propositions}
\crefname{lemma}{Lemma}{Lemmas}
\Crefname{fact}{Fact}{Facts}

\creflabelformat{ineq}{#2{\upshape(#1)}#3}

\newcommand{\iid}{\text{i.i.d.}\xspace}

\newcommand{\lb}{\left[}
\newcommand{\rb}{\right]}
\newcommand{\lp}{\left(}
\newcommand{\rp}{\right)}
\newcommand{\lbr}{\left\{}
\newcommand{\rbr}{\right\}}

\newcommand{\dkl}{d_{\text{KL}}}

\newcommand{\thetahat}{\widehat{\theta}}

\newcommand{\Phat}{\widehat{P}}

\newcommand{\qtext}[1]{\quad \text{#1} \quad}

\newcommand{\phat}{\widehat{p}}

\newcommand{\Reg}{\mathrm{Regret}}
\newcommand{\UIevalue}{W^{\mathrm{UI}}}
\newcommand{\Wmix}{W^{\mathrm{mix}}}

\newcommand{\Vbar}{\overline{V}}
\newcommand{\process}[1]{\langle #1 \rangle}
\newcommand{\Wtilde}{\widetilde{W}}
\newcommand{\Etilde}{\widetilde{E}}

\begin{document}

\maketitle

\begin{abstract}

We consider the  problem of designing optimal level-$\alpha$ power-one tests for composite nulls. Given a parameter $\alpha \in (0,1)$ and a stream of $\mathcal{X}$-valued observations $\{X_n: n \geq 1\} \simiid P$, the goal is to design a level-$\alpha$ power-one test $\tau_\alpha$ for the null $H_0: P \in \mathcal{P}_0 \subset \mathcal{P}(\mathcal{X})$. Prior works have shown that any such $\tau_\alpha$ must satisfy $\mathbb{E}_P[\tau_\alpha] \geq \tfrac{\log(1/\alpha)}{\gamma^*(P, \mathcal{P}_0)}$, where $\gamma^*(P, \calP_0)$ is the so-called $\mathrm{KL}_{\inf}$ or minimum divergence of $P$  to the null class. In this paper, our objective is to develop and analyze constructive schemes that match this lower bound as $\alpha \downarrow 0$. 

We first consider the finite-alphabet case~($|\calX| = m < \infty$), and show that a test based on \emph{universal} $e$-process~(formed by the ratio of a universal predictor and the running null MLE) is optimal in the above sense. The proof relies on a Donsker-Varadhan~(DV) based saddle-point representation of $\mathrm{KL}_{\inf}$, and an application of Sion's minimax theorem. This characterization motivates a general method for arbitrary $\mathcal{X}$: construct an $e$-process based on the empirical solutions to the saddle-point representation over a sufficiently rich class of test functions. We give sufficient conditions for the optimality of this test for compact convex nulls, and verify them for H\"older smooth density models. We end the paper with a discussion on the computational aspects of implementing our proposed tests in some practical settings. 
\end{abstract}

\tableofcontents
\section{Introduction}
\label{sec:itroduction}
Let $(\calX, \calB)$ denote some measurable space, and $\calP_j = \{P_\theta: \theta \in \Theta_j\}$ for $j=0, 1$, denote a pair of disjoint  classes of probability distributions on $(\calX, \calB)$.  Then, given a stream of observations $\{X_n: n\geq 1\} \overset{\iid}{\sim} P$ for an unknown distribution $P \in \calP_0 \cup \calP_1 \subset \calP(\calX)$, we consider the task of  distinguishing between 
\begin{align}
    H_0: P \in \calP_0, \qtext{versus} H_1: P \in \calP_1. \label{eq:testing-problem-def-general}
\end{align}
Given a parameter $\alpha \in (0, 1)$, our goal is to construct a level-$\alpha$ power-one test $\tau_\alpha$ for this problem. Formally, this involves constructing a stopping time $\tau_\alpha$, such that the following two conditions hold: 
\begin{align}
    \sup_{P_0 \in \calP_0} \PP_{P_0}\lp \tau_\alpha < \infty \rp \leq \alpha, \qtext{and}\inf_{P_1 \in \calP_1} \PP_{P_1}\lp \tau_\alpha < \infty \rp  = 1. 
\end{align}
In other words, $\tau_\alpha$ denotes the random data-driven time at which we stop collecting more observations, and reject the null. In addition to the power-one property, it is often important to control the expected value of the stopping time for any $P \in \calP_1$. Prior works have established a fundamental impossibility result on the value that this quantity can take in the limit of small $\alpha$. In particular, it is known that any level-$\alpha$ test $\tau_\alpha$ must satisfy 
\begin{align}
    \liminf_{\alpha \downarrow 0} \frac{\EE_P[\tau_\alpha]}{\log(1/\alpha)/\gamma^*(P, \calP_0)} \geq 1, \qtext{where} \gamma^*(P, \calP_0) = \inf_{P_0 \in \calP_0} \dkl(P \parallel P_0), 
\end{align}
and $\dkl(P \parallel P_0)$ denotes the relative entropy (or KL divergence) between $P$ and $P_0$. The term $\gamma^*(P, \calP_0)$ is the appropriate measure of difficulty of distinguishing $P$ from the null class $\calP_0$. In this paper, our main objective is to develop a  general constructive scheme to design tests that match this fundamental limit. 

The rest of the paper is organized as follows. We formally describe the problem and present some key definitions in~\Cref{subsec:preliminaries}. We discuss the prior work in~\Cref{subsec:related-work} and also present an overview of our main results. In~\Cref{sec:finite-alphabet}, we analyze the performance of a test based on universal $e$-processes and establish its optimality for finite alphabets. Next, we build upon the insights developed from the finite alphabet case to develop a class of optimal tests for arbitrary weakly compact nulls over general alphabets  in~\Cref{sec:general}. Our construction assumes the ability to compute certain functionals, and in~\Cref{subsec:computational-aspects} we discuss situations in which these quantities can be approximated sufficiently accurately in a computationally feasible manner, and also discuss principled heuristics for more general problems. We end the paper with a short discussion of interesting future directions in~\Cref{sec:conclusion}. Lengthy technical proofs and additional background results are deferred to the appendix.

\subsection{Preliminaries}
\label{subsec:preliminaries}
Throughout this paper, we will work in an unspecified probability space that is large enough to allow the definition of all random variables that arise in our setup. As mentioned before, we assume that we have a stream of \iid observations $\{X_n: n \geq 1\}$ drawn from some unknown distribution $P$. Let $\{\calM_n: n \geq 0\}$ denote a filtration such that $\sigma(X_1, \ldots, X_n) \subset \calM_n$ for all $n \geq 1$. This enlargement of the filtration is to incorporate any additional randomness in the design of the test.  We consider a class of probability distributions $\{P_\theta: \theta \in \Theta\}$ on $\calX$, indexed by some parameter set $\Theta$. Note that we do not require $\Theta$ to be finite dimensional, and in the  most general case, elements of $\Theta$  can simply be the probability measures themselves. We define the null and alternative classes $\calP_0$ and $\calP_1$ as distributions indexed by two disjoint subsets $\Theta_0 \subset \Theta$ and $\Theta_1 \subset \Theta$ respectively. Thus, assuming that $P \equiv P_\theta$ for an unknown $\theta \in \Theta_0 \cup \Theta_1$, we can restate the hypothesis testing problem as 
\begin{align}
    \text{Given } \{X_n: n \geq 1\} \overset{\iid}{\sim} P \equiv P_\theta, \quad \text{decide between} \quad H_0: \theta \in \Theta_0, \qtext{vs.} H_1: \theta \in \Theta_1. \label{eq:parametrized-testing-problem-def}
\end{align}
Depending on the specific problem instance, working with this parametrized formulation may be more natural than the equivalent setup of~\eqref{eq:testing-problem-def-general}. 

Our goal is to design a level-$\alpha$ power one test for this problem for a pre-specified value $\alpha \in (0,1)$. Formally, this means that we need to design a stopping time $\tau_\alpha$, such that 
\begin{align}
    \sup_{\theta_0 \in \Theta_0} \PP_{\theta_0}\lp \tau_\alpha < \infty \rp \leq \alpha \qtext{and} \inf_{\theta_1 \in \Theta_1} \PP_{\theta_1}\lp \tau_\alpha < \infty \rp =1. 
\end{align}
In addition, for any $\theta \not \in \Theta_0$, we also want to ensure a small $\EE_{\theta}[\tau_\alpha]$. Prior work, such as~\citet{agrawal2025stopping}, have established an information theoretic lower bound on the expected value of any collection of tests $\process{\tau_\alpha} \coloneqq \{\tau_\alpha: \alpha \in (0, 1)\}$ for a given null $\calP_0$, and we recall the statement next. 
\begin{fact}[Lower Bound]
    \label{fact:lower-bound} Let $\process{\tau_\alpha} \coloneqq \{\tau_\alpha: \alpha \in (0,1)\}$ denote a collection of level-$\alpha$ power-one tests for the composite null $\calP_0$. Then, for any fixed $\alpha \in (0,1)$, and $P \not \in \calP_0$ with $\gamma^*(P, \calP_0)>0$, we have 
    \begin{align}
         \frac{\EE_{P}[\tau_\alpha]}{J(\alpha, P, \calP_0)} \geq 1, \qtext{where} J(\alpha, P, \calP_0) \coloneqq \frac{\log(1/\alpha)}{\gamma^*(P, \calP_0)}.  
    \end{align}
    On taking $\alpha$ to zero, this implies that $\liminf_{\alpha \downarrow 0} \EE_P[\tau_\alpha]/J(\alpha, P, \calP_0) \geq 1$. 
\end{fact}
As illustrated further in~\Cref{fig:data-processing}, this result is a consequence of the data-processing inequality~(DPI) for stopped observation processes. In particular, consider a stream of observations $\{X_n: n \geq 1\} \simiid P_X$ with $P_X \in \{P, Q\}$, with $Q \in \calP_0$, $P \not \in \calP_0$, and $P \ll Q$. Let $\tau_\alpha$ denote any level-$\alpha$ power-one test for this problem, with $\EE_P[\tau_\alpha] < \infty$. Then, we can introduce a $\calM_{\tau_\alpha}$-measurable function $\psi:\lp \lp \cup_{m \in \mathbb{N}}\{m\}\times \calX^m \rp \cup \lp \{\infty\} \times \calX^{\mathbb{N}} \rp \rp  \to \{0, 1\}$ with $\psi(m, X_1, \ldots, X_m) = \boldsymbol{1}_{m< \infty}$, and observe that the Bernoulli random variable $\psi(\tau_\alpha, X^{\tau_\alpha})$ satisfies $\EE_Q[\psi(\tau_\alpha, X^{\tau_\alpha})] \leq \alpha$~(due to the level-$\alpha$ property), and $\EE_P[\psi(\tau_\alpha, X^{\tau_\alpha})] = 1$~(due to the power-one property). Then, as illustrated in~\Cref{fig:data-processing}, DPI implies that the relative entropy between $\mathrm{Law}_P(\tau_\alpha, X_1, \ldots, X_{\tau_\alpha})$ and $\mathrm{Law}_Q(\tau_\alpha, X_1, \ldots, X_{\tau_\alpha})$ cannot be smaller than the relative entropy between the corresponding outputs~(i.e., $\psi(\tau_\alpha, X^{\tau_\alpha})$), which is lower bounded by $\log(1/\alpha)$. Finally, we can apply Wald's identity to show that the larger relative entropy is equal to $\EE_P[\tau_\alpha] \dkl(P \parallel Q) \geq \log(1/\alpha)$, and minimizing this over all $Q \in \calP_0$ gives us the result recalled in~\Cref{fact:lower-bound}. 
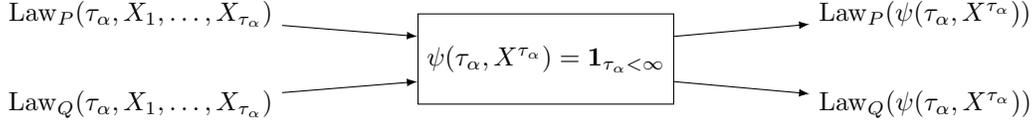
\begin{figure}[t!]
    \centering
\begin{tikzpicture}[>=latex]
  \node[draw, minimum width=1.8cm, minimum height=1.2cm] (psi) {$\psi(\tau_\alpha, X^{\tau_\alpha}) = \boldsymbol{1}_{\tau_\alpha < \infty}$};

  \node[left=1.8cm of psi, yshift=0.6cm]  (PX) {$\mathrm{Law}_P(\tau_\alpha, X_1, \ldots, X_{\tau_\alpha})$};
  \node[left=1.8cm of psi, yshift=-0.6cm] (QX) {$\mathrm{Law}_Q(\tau_\alpha, X_1, \ldots, X_{\tau_\alpha})$};
  \node[right=1.8cm of psi, yshift=0.6cm] (PY) {$\mathrm{Law}_P( \psi(\tau_\alpha, X^{\tau_\alpha}))$};
  \node[right=1.8cm of psi, yshift=-0.6cm] (QY) {$\mathrm{Law}_Q(\psi(\tau_\alpha, X^{\tau_\alpha}))$};

  \draw[->] (PX) -- (psi.170);
  \draw[->] (QX) -- (psi.190);
  \draw[->] (psi.10)  -- (PY);
  \draw[->] (psi.350) -- (QY);
\end{tikzpicture}
\caption{
The diagram above illustrates the data processing inequality for relative entropy~\citep[Theorem 2.17]{polyanskiy2025information}, which says that the distance between two input distributions~(as measured by $\dkl$) can never be smaller than the distance between the corresponding output distributions, after passing through a common Markov kernel. In our case, the channel corresponds to the decision rule $\boldsymbol{1}_{\tau_\alpha < \infty}$, which naturally leads to the conclusion recalled in~\Cref{fact:lower-bound}, since it implies that $\EE_P[\tau_\alpha] \dkl(P \parallel Q) \geq \log(1/\alpha)$ for any level-$\alpha$ power-one test $\tau_\alpha$, and for distributions $Q \in \calP_0$ and $P \not \in \calP_0$. 
}
\label{fig:data-processing}
\end{figure}

The impossibility result recalled in~\Cref{fact:lower-bound} then naturally leads to the following definition of ``optimal'' collection of tests. 
\begin{definition}[First-Order Optimality]
    \label{def:first-order-optimality} Given a null class $\calP_0$, we say that a collection of tests $\process{\tau_\alpha} \coloneqq \{\tau_\alpha: \alpha \in (0,1)\}$ for the null $\calP_0$ is \emph{first order-optimal}~(or simply optimal) if it satisfies the following for any $P \in \calP_1$, with $\gamma^*(P, \calP_0) = \inf_{Q \in \calP_0} \dkl(P \parallel Q)> 0$: 
    \begin{align}
        \limsup_{\alpha \downarrow 0} \frac{\EE_P[\tau_\alpha]}{J(\alpha, P, \calP_0)} \leq 1, \qtext{where} J(\alpha, P, \calP_0) = \frac{\log(1/\alpha)}{\gamma^*(P, \calP_0)}. 
    \end{align}
    This condition, combined with the lower bound of~Fact~\ref{fact:lower-bound} implies that this limit must be equal to $1$. 
\end{definition}

The focus of this paper is to design and analyze power-one tests that are optimal in the sense defined above. Before proceeding to the main results, we present a brief discussion of related works in this area that will provide the appropriate context to place our contributions. 

\subsection{Related Work and Overview of Our Results}
\label{subsec:related-work}
The study of power-one level-$\alpha$ sequential tests was pioneered by Robbins and collaborators in the 1960s and 70s. For example, for the one-sided Gaussian mean testing problem with $\process{X_n} = \{X_n: n \geq 1\} \simiid N(\mu, 1)$ and $H_0: \mu \leq 0$ versus $H_1: \mu > 0$, \citet{darling1967iterated} proposed  tests  that continue sampling indefinitely with probability at least $1-\alpha$ under $H_0$, while stopping almost surely under $H_1$. Their construction used moment generating function~(MGF) based martingales combined with a ``peeling'' argument that led to an iterated-logarithm boundary crossing of the running sum $\sum_{i=1}^n X_i$. 
In a follow-up work,~\citet{darling1968some} considered some nonparametric goodness-of-fit and stochastic dominance testing tasks, and developed sequential power-one variants of classical Kolmogorov-Smirnov (KS) style tests. Their construction essentially relied on a combinatorial deviation inequality for the KS statistic at any fixed $n$, combined with a simple union bound. They showed the level-$\alpha$ and power-one property, and also obtained upper bounds on the expected stopping times of the resulting tests in terms of the effect-size and the inverse of the rejection boundary. Some early follow-up works include~\citet{robbins1970statistical}, who developed tests based on mixture of likelihood-ratios~(and more generally test martingales) for various parametric and nonparametric problems, and \citet{robbins1972class} who discussed an alternative approach based in generalized-likelihood ratio process for parametric problems, with predictable estimates in the numerator and a fixed~(extremal) parameter in the denominator.

More recently, there has been a resurgence of interest in this area of anytime-valid inference over the last decade, and a detailed discussion can be found in the survey paper by~\citet{ramdas2023game} and the monograph by~\citet{ramdas2024hypothesis}. Recent works design and analyze new power-one tests in both the parametric~(e.g., \citet{grunwald2024safe, perez2024statistics}) and nonparametric~(e.g., \citet{waudby2024estimating, shekhar2023nonparametric, podkopaev2023sequential, henzi2024rank}) settings. These works often involve bounds on the expected stopping time under the alternative, and establish some notions of optimality. For instance, in the case of two-sample testing,~\citet{shekhar2023nonparametric} show the existence of ``hard'' problems for which the performance achieved by their test cannot be improved. In this paper, we are interested in analyzing the instance-wise optimality of level-$\alpha$ tests in the limit of small $\alpha$, as we discussed in~\Cref{def:first-order-optimality}. The most closely related works that address similar questions are~\citep{agrawal2025stopping} and~\citep{waudby2025universal}, and we discuss them below. 

\citet{waudby2025universal} studied the task of analyzing the performance under the alternative of a particular class of power-one tests, associated with $e$-processes of the form 
\begin{align}
    W_0 = 1, \quad W_n = W_{n-1} \times \lp \lambda_n E_n^{(1)} + (1-\lambda_n) E_n^{(2)} \rp, 
\end{align}
where $\{E_n^{(j)}: n \geq 1\}$ for $j=1,2$ are a collection of \iid $e$-values under the null, and $\process{\lambda_n} \coloneqq \{\lambda_n: n \geq 1\}$ denotes a sequence of predictable ``bets''. This process can be interpreted as the wealth of an investor~(who starts off with $W_0=\$1$) in a two-stock portfolio with price-relatives $E_n^{(1)}$ and $E_n^{(2)}$ in the $n^{th}$ investing period, and $(\lambda_n, 1-\lambda_n)$ denotes the rebalancing done by the investor at the start of the $n^{th}$ period. One of the results proved by~\citet{waudby2025universal} is that if the betting~(or investing) strategy $\process{\lambda_n}$ has vanishing per-sequence regret, then the expected stopping time of the resulting test satisfies $\limsup_{\alpha \downarrow 0} \EE[\tau_\alpha]/\log(1/\alpha) \leq 1/\sup_{\lambda \in [0,1]} \EE[\log(\lambda E^{(1)} + (1-\lambda) E^{(2)})]$ under certain mild moment conditions~\citep[Theorem 3.3]{waudby2025universal}. The authors also show that this dependence cannot be improved for this particular class of tests. Our work differs from~\citet{waudby2025universal} in scope as well as technical details: instead of working with a given class of $e$-values or processes which implicitly define the null, our goal is to start with a given composite null class and construct appropriate power-one tests that are optimal. The notion of optimality we care about is w.r.t. all feasible power-one tests, and not the restricted class that was the focus of~\citet{waudby2025universal}.

Another related paper studying the fundamental limits of power-one tests is~\citep{agrawal2025stopping}.  They consider a general testing problem with a composite null $\calP_0$, and show that any level-$\alpha$ test $\tau_\alpha'$ must satisfy $\EE_P[\tau_\alpha'] \geq \log(1/\alpha)/\gamma^*(P, \calP_0)$, as we recalled in~\Cref{fact:lower-bound}. Next, they show that for tests of the form $\tau_\alpha \coloneqq \inf \{n \geq 1: W_n \geq 1/\alpha\}$ with $\process{W_n}$ denoting an $e$-process for the null $\calP_0$, \citet[Theorem B.1]{agrawal2025stopping} showed that 
\begin{align}
    \liminf_{n \to \infty} \frac{1}{n} \log W_n \; \stackrel{\text{a.s.}}{\geq}\; \gamma^*(P, \calP_0) \quad \implies \quad \limsup_{\alpha \downarrow 0} \frac{\EE_P[\tau_\alpha]}{\log(1/\alpha)/\gamma^*(P, \calP_0)} \leq 4. 
\end{align}
Our contributions are complementary to~\citet{agrawal2025stopping}.  Rather than identifying sufficient conditions on $e$-processes for the associated tests to be optimal, we  develop  general design principles, via DV representation and minimax theorem, for building $e$-processes~(and hence, power-one tests) that achieve the optimal lower bound as $\alpha \downarrow 0$ with a sharp constant~(note that this constant factor refinement is a consequence of the alternative  analysis route developed in our paper, and not the primary driving goal of the paper).

\paragraph{Overview of our results.} We begin in~\Cref{sec:finite-alphabet} by considering the technically simpler finite-alphabet setting, where the null is represented by a closed subset $\Theta_0$ of the simplex of probability mass functions. We analyze a test $\tau_\alpha$ defined as the first $1/\alpha$ crossing of a \emph{universal $e$-process}, constructed as a ratio of a universal code~(e.g., a Krichevsky-Trofimov code or equivalently, a $\mathrm{Dirichlet}(1/2, \ldots, 1/2)$ mixture) in the numerator and the running maximum likelihood over the null $\Theta_0$ in the denominator. Our first main  result,~\Cref{theorem:UI-test-finite}, shows that this test is optimal in the limit of small $\alpha$. 

A key observation underlying our proof of~\Cref{theorem:UI-test-finite} is a saddle point characterization of $\gamma^*(P_\theta, \calP_0) \equiv \gamma^*(\theta, \Theta_0)$. In particular,  using the DV-representation of relative entropy~(\Cref{fact:donsker-varadhan}) and Sion's minimax theorem~(\Cref{fact:Sions-Minimax-Theorem}), we obtain a min-max~(equivalent max-min) representation that identifies an ``oracle'' witness function that achieves the optimal value in $\gamma^*$. This perspective motivates our general $e$-process construction~(\Cref{def:DV-e-process}) for arbitrary alphabets based on a predictable empirical-risk-minimization~(ERM) based choice of test functions from a class $\calF$ that approximate $f^*$. Under certain conditions stated in~\Cref{assump:general}, we establish the optimality of tests associated with such $e$-processes in the sense of~\Cref{def:first-order-optimality}. As a concrete nonparametric example, we discuss an application of our result to a problem in which the null $\calP_0$ is a compact convex class of  H\"older smooth densities over a compact domain, and the test functions lie in an RKHS associated with a universal kernel. 

We end the paper with a brief discussion of the practical aspects associated with implementing our proposed $e$-processes in~\Cref{subsec:computational-aspects}. In particular, for certain finite dimensional problems, we can employ existing results from stochastic saddle-point optimization to obtain computationally feasible tests using our developed methodology. For larger problems, involving complex and high-dimensional data types~(such as images, or videos), we briefly discuss a heuristic strategy that can directly leverage the well-established training pipelines of generative machine learning~(ML) models, such as $f$-GANs. 

\section{Finite Alphabet Case}
\label{sec:finite-alphabet}
In this section, we work with the simplifying assumption that $\calX = \{x_1, \ldots, x_m\}$ denotes a finite alphabet. In this setting, the class of probability distributions $\calP(\calX)$ is parametrized by the $(m-1)$-dimensional simplex $\Theta = \{\theta \in [0,1]^m: \sum_{j=1}^m \theta[j] = 1\}$. Let $\calP_0 \subset \calP(\calX)$ denote a class of distributions indexed by a convex, compact subset $\Theta_0 \subset \Theta$; that is, $\calP_0 = \{P_\theta: \theta \in \Theta_0\}$, where $P_\theta$ is the probability measure on $\calX$ associated with the probability mass function $\theta$. Given $\{X_n: n\geq 1\} \simiid P_\theta$ and a confidence parameter $\alpha \in (0, 1]$, our goal is to construct a level-$\alpha$ power-one test to decide between 
\begin{align}
    H_0: \theta \in \Theta_0, \qtext{versus} H_1: \theta \in \Theta_1,  \label{eq:composite-test-finite-alphabet}
\end{align}
for another subset $\Theta_1 \subset \Theta$ that is disjoint with $\Theta_0$. 

For any pmf $\theta \in \Theta$, and a sequence $x^n = (x_1, \ldots, x_n)$, we will use the shorthand $\theta^n[x^n]$ to represent $\prod_{i=1}^n \theta[x_i]$. For a distribution $P \equiv P_\theta \in \calP(\calX)$ with $\theta \not \in \Theta_0$, we will alternate between the notations $\gamma^*(P, \calP_0)$ and $\gamma^*(\theta, \Theta_0)$ in this section. Both terms will be used to refer to the minimum divergence of $\theta$~(or $P_\theta$) to the null $\Theta_0$~(or $\calP_0$): 
\begin{align}
\gamma^* \equiv \gamma^*(\theta, \Theta_0) \equiv \gamma^*(P_\theta, \calP_0) = \inf_{\theta_0 \in \Theta_0} \dkl(P_\theta \parallel P_{\theta_0}), \qtext{and} \theta_0^\dagger \equiv \theta_0^\dagger(\theta) \in \argmin_{\theta_0 \in \Theta_0}\dkl(P_\theta \parallel P_{\theta_0}). 
\end{align}
Since we have assumed that $\Theta_0$ is a compact subset of $\Theta$, it implies that the infimum in the definition of $\gamma^*(\theta, \Theta_0)$ is achieved, and the term $\theta^\dagger(\theta)$ is well defined.  Within this context, our goal is to construct  level-$\alpha$ stopping times $\{\tau_\alpha: \alpha \in (0, 1]\}$ that are provably optimal in the sense discussed in~\Cref{def:first-order-optimality}. 

\paragraph{Construction of the test.} A general strategy for constructing power-one tests is to define them as the first $1/\alpha$ crossing of a process $\process{W_n} \coloneqq \{W_n: n \geq 0\}$ adapted to the filtration $ \process{\calM_n} \coloneqq \{\calM_n: n \geq 0\}$, satisfying the following two conditions: 
\begin{itemize}
    \item Under $H_0$, if $\theta_0 \in \Theta_0$ is the true parameter, then the process $\{W_n: n \geq 1\}$ is unlikely to take large values. 
    \item Under $H_1$, if $\theta_1 \in \Theta_1$ is the true parameter, then the $\{W_n: n \geq 1\}$ grows rapidly at an exponential rate, that is approximately equal to $\gamma^*(\theta_1, \Theta_0)$. 
\end{itemize}
In other words, $\process{W_n}$  quantifies the accumulated evidence against the null based on the observations $\process{X_n}$. 
If $\Theta_0$ and $\Theta_1$ were singletons $\{\theta_0\}$ and $\{\theta_1\}$, then the optimal choice  would be the likelihood ratio  process with $W_n = L_n(X^n; \theta_1, \theta_0) = \prod_{i=1}^n \theta_1^n[X^n]/\theta_0^n[X^n]$. For the case of a point null $\Theta_0 = \{\theta_0\}$ but composite alternative $\Theta_1$, a natural idea is to take mixtures 
\begin{align}
    \Wmix_n = \int_{\Theta} L_n(X^n; \theta_1, \theta_0) \pi_J(\theta) d\theta = \frac{\int_{\Theta} \theta^n[X^n] \pi_J(\theta) d\theta}{\theta_0^n[X^n]}, 
\end{align}
where $\pi_J$ denotes the Jeffreys prior over the probability simplex $\Theta$. The reason for this particular formulation is that classical results from universal compression imply that under any $\theta \in \Theta_1$, and $n \geq 1$, we have 
\begin{align}
    \log \Wmix_n \geq \log L_n(X^n; \theta, \theta_0) - \frac {m-1}2 \log n - C_m, 
\end{align}
for some constant $C_m>0$. In other words, under any $\theta \in \Theta_1$,  the mixture process $\{\Wmix_n: n \geq 1\}$ grows essentially as quickly as the optimal likelihood ratio process modulo a small logarithmic regret term. 
Finally,  to deal with the most general case of composite null and alternative, we  define $\UIevalue$ as 
\begin{align}
    \UIevalue_n = \frac{\text{Mixture over all $\theta$}}{\text{Supremum over all null parameters}} = \frac{ \int_{\Theta} \theta^n[X^n] \pi_J(\theta) d\theta}{\sup_{\theta_0 \in \Theta_0} \theta_0^n[X^n]}. 
\end{align}
This simple idea, introduced by~\citet[Section 8]{wasserman2020universal}, ends up giving us an (as we will show later)  optimal test for the general setting of composite null and alternative. Using this process, given the stream of observations $\process{X_n}$, we define our level-$\alpha$ test as 
\begin{align}
    \tau_\alpha = \inf \{n \geq 1: \UIevalue_n \geq 1/\alpha\}, \qtext{with} \UIevalue_0 = 1, \; \text{and}\; \UIevalue_n = \frac{\theta_J^n[X^n]}{\sup_{\theta_0 \in \Theta_0} \theta_0[X^n] }.  \label{eq:UI-eprocess}
\end{align}
Here, $\theta_J^n$ denotes the mixture distribution over $\calX^n$ with respect to Jeffreys prior, i.e., the Krichevsky-Trofimov code~\citep{krichevsky1981performance}, and observe that the denominator is simply the maximum likelihood value based on the first $n$ observations: 
\begin{align}
\theta_J^n[x^n] = \int_{\Theta} \theta^n[x^n] \pi_J(\theta) d\theta, \qtext{and} \thetahat_{0}^n[x^n] = \sup_{\theta_0 \in \Theta_0}  \theta_0^n[x^n].  
\end{align}
\begin{remark}
    Using the properties of $\mathrm{Dirichlet}(1/2, \ldots, 1/2)$ distribution~(which is the Jeffreys prior in this case), it follows that the mixture $\theta_J^n$ can be written as 
    \begin{align}
        \theta_J^n[x^n] =  \prod_{i=1}^n \theta_J[x_i \mid x^{i-1}], \qtext{with} \theta_J[x \mid x^{i-1}] = \frac{ 1/2 + \sum_{j=1}^{i-1} \boldsymbol{1}_{x_j=x} }{ m/2 + i-1}. 
    \end{align}
    That is, each conditional $\theta_J[\cdot \mid x^{i-1}]$ is the so-called ``add-$1/2$'' predictor, where we simply add a fictitious count of $1/2$ to each $x \in \calX$ to smooth out our estimate. Thus, the numerator in~\eqref{eq:UI-eprocess} can be updated incrementally in a computationally efficient manner. 
    However, the denominator requires recomputing the maximum likelihood value from scratch in each round for most problems, meaning that evaluating $\UIevalue_n$ can possibly end up being  computationally demanding depending on the definition of $\Theta_0$.  Since our focus in this paper is on the statistical performance of the test, we do not discuss the computational issues further. 
\end{remark}
We now proceed to the analysis of the tests $\process{\tau_\alpha} = \{\tau_\alpha: \alpha \in (0, 1]\}$ defined in~\eqref{eq:UI-eprocess}. 

\begin{theorem}
    \label{theorem:UI-test-finite} For the testing problem with composite null  in~\eqref{eq:composite-test-finite-alphabet}, the class of tests $\process{\tau_\alpha}$ defined in~\eqref{eq:UI-eprocess} satisfy the following for any $\theta \not \in \Theta_0$ such that $\gamma^*(\theta, \Theta_0) \coloneqq \inf_{\theta_0 \in \Theta_0} \dkl(\theta \parallel \theta_0) > 0$: 
    \begin{align}
        &\sup_{\theta_0 \in \Theta_0} \mathbb{P}_{\theta_0}(\tau_\alpha< \infty) \leq \alpha, \; \;\forall \alpha \in (0, 1], \qtext{and}
          \limsup_{\alpha \downarrow 0} \frac{\mathbb{E}_{\theta}[\tau_\alpha]}{J(\theta, \Theta_0, \alpha)}  \leq 1, \qtext{where} J(\theta, \Theta_0, \alpha) = \frac{\log(1/\alpha)}{\gamma^*(\theta, \Theta_0)}. 
    \end{align}
     This establishes the optimality of the test $\process{\tau_\alpha}$ in the sense of~\Cref{def:first-order-optimality}.  Additionally, note that the finiteness of the expected stopping times also implies the weaker power-one property of $\process{\tau_\alpha}$.
\end{theorem}
We present a detailed proof of this result in~\Cref{proof:UI-test}. An important fact about the process $\{\UIevalue_n: n \geq 0\}$ is that it is neither a martingale nor a supermartingale under the null. Yet, as noted by~\citet{wasserman2020universal}, we can still employ Ville's inequality on this process to control the type-I error because, for every null $\theta_0 \in \Theta_0$,  there exists a nonnegative supermartingale~(or actually a nonnegative martingale) that dominates the process $\{\UIevalue_n: n \geq 0\}$ pointwise. This is one of the defining properties of \emph{e-processes}~\citep[Lemma 6]{ramdas2020admissible}, which are fundamental tools in sequential anytime-valid inference.  

\section{Extension to General Alphabets}
\label{sec:general}

A closer look at the proof of~\Cref{theorem:UI-test-finite}~(detailed in~\Cref{proof:UI-test}) reveals that, in addition to relying on the finiteness of the alphabet $\calX$, the argument also relies strongly on two things: (i) the (weak) compactness of the null class, along with (ii) the application of Sion's minimax theorem.  Together, these two facts allow an interchange the $\inf$ and $\sup$ in the definition of $\gamma^*(\theta, \Theta_0)$: 
\begin{align}
    \gamma^*(\theta, \Theta_0) = \inf_{\theta_0 \in \Theta_0} \sup_{f:\calX \to \R} \EE_{\theta}[f(X)] - \log \EE_{\theta_0}[e^{f(X)}] \;=\; \sup_{f:\calX \to \R} \inf_{\theta_0 \in \Theta_0}  \EE_{\theta}[f(X)] - \log \EE_{\theta_0}[e^{f(X)}]. \label{eq:variational-divergence-to-null}
\end{align}
This, in turn, naturally leads to the definition of the ``oracle $e$-process'' $\{\prod_{i=1}^n \exp \lp f^*(X_i) - \psi_0(f^*) \rp: n \geq 1\}$, with $f^*$ denoting a function achieving the supremum.  We can use this as a starting point to construct a new class of $e$-processes for general alphabets and establish their optimality under certain conditions. In this section, we first design and analyze such $e$-processes and then discuss their computational feasibility in~\Cref{subsec:computational-aspects}. 

Throughout this section, we work with a measurable space $(\calX, \calB)$, where $\calX$ is a metric space endowed with a metric $\rho$, and $\calB$ is the corresponding Borel sigma algebra. While the metric structure is not strictly necessary in our argument, it allows some simplification of the proof of our main result since the DV representation of relative entropy between distributions on such spaces can be defined using continuous bounded functions~(\Cref{fact:donsker-varadhan}). 
Let $\calP_0 = \{P_{\theta_0}: \theta_0 \in \Theta_0\}$ denote the  class of null distributions on $\calX$ indexed by some parameter set $\Theta_0$. Let $\calF:\calX \to \R$ denote a convex family of functions on $\calX$, and define 
\begin{align}
    d_\calF(P \parallel \calP_0) = \sup_{f \in \calF} \,\lbr \EE_P[f(X)] - \psi_0(f) \rbr, \qtext{where} \psi_0(f) = \sup_{Q \in \calP_0} \log \lp \EE_{Q}[e^{f(X)}] \rp. 
\end{align}

We first record our assumptions formally, where we also define the class of alternatives that we consider in this section using the function class $\calF$. 
\begin{assumption}
    \label{assump:general}  We will work under the following assumptions: 
    \begin{itemize}
        \item \textbf{(A1): Null Class.} $\calP_0$ is a weakly compact and convex subset of $\calP(\calX)$. 
        \item \textbf{(A2): Alternative Class.}  For a given function class $\calF:\calX \to \R$ and a null class $\calP_0$, the set of alternative distributions,  $\calP_1 \subset \calP(\calX) \setminus \calP_0$, consists of $P \in \calP(\calX)$ with 
        \begin{align}
            \dkl(P \parallel \calP_0) = \inf_{Q \in \calP_0} \dkl(P \parallel Q) =  d_{\calF}(P \parallel \calP_0) >0. 
        \end{align}
        
        \item \textbf{(A3): Uniform Moment Bound on $\calF$ under the alternative.}  For every $P \in \calP_1$, there exists a constant $B_P < \infty$, and a $\delta>0$, such that 
        \begin{align}
            \sup_{f \in \calF} \EE_P\lb  \left\lvert f^{4+\delta}(X) \right\rvert \rb \leq B_P < \infty. 
        \end{align}

        \item \textbf{(A4): Uniform MGF Bound on $\calF$ under the null.} There exists a constant $B_0 < \infty$, such that 
        \begin{align}
            \sup_{f \in \calF} |\psi_0(f)| = \sup_{f \in \calF} \sup_{\theta_0\in \Theta_0} \left \lvert \log \lp \EE_{\theta_0}[e^{f(X)}] \rp \right \rvert \leq B_0 < \infty. 
        \end{align}
        \item \textbf{(A5): Learnability of $\calF$.} For each $P \in \calP_1$, there exists a sequence $r_n \equiv r_n(P, \calF) \to 0$, such that for all $n \geq 1$, we have the following, with the same $\delta>0$ from~\textbf{(A3)}: 
        \begin{align}
            \PP_P\lp \sup_{f \in \calF} \left\lvert \frac 1n  \sum_{i=1}^n f(X_i) - \EE_P[f(X)] \right\rvert \leq r_n \rp     \geq 1 - \frac{1}{n^{2 + \tfrac{\delta}{3}}}, 
        \end{align}
        where $X_1, X_2, \cdots$ are \iid draws from $P$.  
    \end{itemize}
\end{assumption}
Some of these conditions, such as \textbf{(A4)}, are not strictly necessary and can be weakened. We use the above statements as a balance between generality and simplicity of the proofs. Next, we propose a sequential test based on the variational representation of the divergence to null term in~\eqref{eq:variational-divergence-to-null}, which uses a sequence of ``greedy'' or empirical risk minimization~(ERM) based approximations of $f^*$. 
\begin{definition}[DV $e$-process]
    \label{def:DV-e-process}
    Consider a function class $\calF$, and a null hypothesis represented by $\calP_0 = \{P_\theta: \theta \in \Theta_0\}$. Given a stream of observations $\{X_n: n \geq 1\} \overset{\iid}{\sim} P \in \calP_0 \cup \calP_1$, we can define the process $\process{W_n} \coloneqq \{W_n: n \geq 0\}$, with $W_0 = 1$, $f_1$ an arbitrary element of $\calF$,  and
    \begin{align}
        W_n = W_{n-1} \times \exp \lp f_n(X_n) - \psi_0(f_n) \rp, \qtext{with} f_n \in \argmax_{f \in \calF} \; \frac 1{n-1} \sum_{i=1}^{n-1} f(X_i) - \psi_0(f), \quad \text{for } n\geq 2.
    \end{align}
    This can be used to define a level-$\alpha$ test for $\calP_0$ as $\tau_\alpha = \inf \{n \geq 1: W_n \geq 1/\alpha\}$ for all $\alpha \in (0, 1]$. 
\end{definition}
In this section, we will establish that under~\Cref{assump:general}, the tests defined above are optimal in the sense of~\Cref{def:first-order-optimality}. We touch upon some computational issues of implementing such tests in~\Cref{subsec:computational-aspects}, and also discuss two important instantiations~(for finite alphabets, and for testing means of bounded observations) in~Appendix~\ref{appendix:DV-instantiations}. 

\begin{remark}
    \label{remark:definition-of-erm} In some cases, the $\argmax$ above may not be well defined or may require appeals to a measurable selection theorem to be made completely rigorous. To avoid such cases, a more robust definition might be to work with an $f_n$ such that $\frac 1{n-1} \sum_{i=1}^{n-1} f_n(X_i) - \psi_0(f_n) \geq \sup_{f \in \calF} \frac 1{n-1} \sum_{i=1}^{n-1} f(X_i) - \psi_0(f) - s_n$, for some real-valued sequence $\process{s_n}$ converging to $0$. In fact, instead of using a specific ERM estimator/predictor as in the definition above, we may generalize the definition for any  predictable sequence $\process{f_n} = \{f_n: n \geq 1\}$, and study conditions on the quality of predictions $\process{f_n}$ to ensure optimality. We leave this investigation for future work. 
\end{remark}

The process $\{W_n: n \geq 0\}$ and the associated tests $\process{\tau_\alpha}$ depend on the ability to compute $\psi_0(f_n) = \sup_{Q \in \calP_0} \log E_{Q}[e^{f(X)}]$, which may be computationally infeasible in some practical problems. We briefly discuss the computational aspects of implementing this scheme in~\Cref{subsec:computational-aspects}, but for now, we focus on analyzing its statistical properties. 
\begin{theorem}
    \label{theorem:general-DV-e-process} 
    Under the conditions stated in~\Cref{assump:general}, for any $\alpha\in (0, 1)$, the test $\tau_\alpha$ introduced in~\Cref{def:DV-e-process}, is level-$\alpha$ for any $Q \in \calP_0$. Furthermore, under $H_1$, for any $P \in \calP_1$, the tests $\process{\tau_\alpha}$ satisfy  
    \begin{align}
        \limsup_{\alpha \downarrow 0} \frac{\mathbb{E}_P[\tau_\alpha]}{\log(1/\alpha)/\gamma^*} \leq 1, \qtext{where} \gamma^* \equiv \gamma^*(P, \calP_0) \coloneqq \inf_{Q \in \calP_0} \dkl(P \parallel Q). 
    \end{align}
\end{theorem}
This result can be interpreted as establishing the ``first-order optimality'' of the test introduced in~\Cref{def:DV-e-process}. Such results have been established for several composite parametric families, however, our scheme also applies in nonparametric settings as we illustrate in the next subsection. We defer the proof of this result to~Appendix~\ref{proof:general-DV-e-process}.

\begin{remark}
    \label{remark:compactness-of-null} The compactness and convexity assumption~\textbf{(A1)} on the null distribution class allows the interchange of infimum and supremum in the definition of the terms $\dkl(P \parallel \calP_0)$ via the Donsker-Varadhan~(DV) variational representation. It holds automatically in parametric models with compact parameter sets, and also in many interesting nonparametric problems as well. We discuss one such example in~\Cref{subsec:example} of uniformly bounded (from above and below) H\"older balls of densities. 
\end{remark}

\begin{remark}
    \label{remark:alt-class} Assumption \textbf{(A2)} says that the function class $\calF$ is rich enough so that the DV representation of the divergence to the null class is not loosened by restricting to $\calF$. In other words, this holds when the log likelihood-ratio between the distribution $P \in \calP_1$, and its reverse information projection $Q^* \in \calP_0$ can be approximated arbitrarily closely by a sequence of functions in $\calF$. In~\Cref{subsec:example}, we present an example where the reproducing kernel Hilbert space~(RKHS) associated with a universal kernel over a compact domain possesses this uniform approximation property. 
\end{remark}

\begin{remark}
\label{remark:moment-bound} The uniform bounded $(4+\delta)$ moment assumption \textbf{(A3)} is a technical assumption; which in conjunction with a moment inequality~(\Cref{fact:pinelis-1994}) for discrete-time martingales gives us a key concentration result used in the proof of~\Cref{theorem:general-DV-e-process}. Similarly, the uniformly bounded MGF condition under the null stated in~\textbf{(A4)} guarantees that the one-step increments of our $e$-process in~\Cref{def:DV-e-process} is well-defined.  
\end{remark}

\begin{remark}
    \label{remark:learnability} The assumption \textbf{(A5)} is a classical uniform law of large numbers assumption that requires the empirical processes indexed by $\calF$ to concentrate at rate $r_n$ with high probability. This condition is satisfied by many ``regular'' function classes, such as the RKHS balls with bounded kernels that we consider in the next subsection. Importantly this condition is much weaker than the sequential learnability conditions, that are necessary to ensure the existence of prediction schemes with pathwise regret bounds. Our proof of~\Cref{theorem:general-DV-e-process} proceeds without requiring the existence of such pathwise conditions. 
\end{remark}

\begin{remark}
    \label{remark:other-learning-schemes} We have made the specific choice of the ERM strategy in~\Cref{def:DV-e-process} and the corresponding learnability assumptions, such as condition~\textbf{(A5)}, to keep our discussion concrete. With a view towards practical implementations, the statement and proof of~\Cref{theorem:general-DV-e-process} can be generalized to allow for the choice of $f_n$ to be an approximate ERM solution with vanishing error~(as $n \to \infty$), or to other schemes for learning  $f_n$ with appropriate modifications to the assumptions. We leave a more thorough investigation of selecting $f_n$, and the minimal conditions needed for the optimality result like~\Cref{theorem:general-DV-e-process}, to future work. 
\end{remark}

\subsection{A Nonparametric Example}
\label{subsec:example}

\begin{definition}
    \label{def:nonparametric-example}
    Throughout this section, we work with the compact alphabet $\calX = [0,1]^d$ for some $d \geq 1$ endowed with the usual Borel sigma algebra. 
    \begin{itemize}
        \item Let the null class $\calP_0 = \{P_\theta: \theta \in \Theta_0\}$ with $\Theta_0 \equiv \Theta_0(c_0, C_0, L, \beta)$ consists of densities~(with respect to the Lebesgue measure $\mu$) bounded away from $0$ and $\infty$; that is, $\Theta_0 = \{\theta:\calX \to [0, \infty]: |\theta(x)-\theta(x')| \leq L \|x-x'\|_2^\beta, \, c_0 \leq |\theta(x)| \leq C_0, \; \text{and}\; \int \theta d \mu = 1\}$. Here, we have assumed $\beta \in (0, 1]$, $L>0$, $0<c_0<C_0<\infty$. 
    
        \item Instead of choosing a single function class $\calF$, we will work with a countable collection $\process{\calF_j} \equiv \{\calF_j: j \geq 0\}$. Our tests will be based on a mixture of DV-$e$-processes associated with these function classes. In particular, for any $j \geq 0$, we set $\calF_j$ to be  a $2^j$ norm ball of an RKHS $\calH_\kappa$ associated with a universal~\citep{micchelli2006universal} positive-definite kernel $\kappa: \calX \times \calX \to \mathbb{R}$. That is, $\calF_j = \{f \in \calH_\kappa: \|f\|_{\kappa} \leq 2^j\}$. Additionally, we assume that the kernel $\kappa$ is uniformly bounded $\sup_{x, x' \in \calX} \kappa(x, x') \leq 1$. 
    
        \item Finally, we consider a  class of alternatives $\calP_1$ defined as 
        \begin{align}
            \calP_1 = \lbr P \in \calP(\calX) \setminus \calP_0: \exists \theta \in \Theta_0, \text{ with } P \ll P_{\theta}, \; \frac{dP}{dP_{\theta}} \in \calC(\calX, \R), \; \text{and} \; \frac{dP}{dP_{\theta}} > R_0 > 0 \rbr.  \label{eq:calP1-example}
        \end{align}
        Here, $\calC(\calX, \R)$ denotes the space of continuous real-valued functions on (the compact domain) $\calX$. 
    \end{itemize}
    With these terms, for any $j \geq 0$, we can define a process  $\process{W_n^j} = \{W_n^j: n \geq 0\}$ following the approach described in~\Cref{def:DV-e-process} using the function class $\calF_j$. These can be combined using weights $\process{c_j} = \{c_j \in (0,1):  j \geq 0\}$, with $c_j = \tfrac{6}{\pi^2 (j+1)^2}$ to define 
\begin{align}
    \tau_\alpha = \inf \{n \geq 1: W_n \geq 1/\alpha\}, \qtext{with} W_n = \sum_{j \geq 0} c_j W_n^j,\quad \text{for all } n \geq 0. \label{eq:example-test}
\end{align}
\end{definition}

\begin{remark}
    \label{remark:nonparam-exmaple}
    The requirement that for every alternative $P$, there exists a null $\theta$, such that the likelihood ratio $dP/dP_\theta$ is bounded away from $0$ is imposed mainly for technical convenience. It ensures that the corresponding log-likelihood ratio is bounded continuous function on $\calX$, which then simplifies the argument of the main result of this section stated below~(\Cref{prop:example}). The statement of~\Cref{prop:example} would remain valid even without this lower-bound condition. The proof, however,  would then require a more delicate approximation argument, for instance, based on the truncation of the possibly unbounded log-likelihood ratios and an additional limiting step. 
\end{remark}
We can now state the main result of this section, which says that the tests defined in~\eqref{eq:example-test} are optimal  in the limit of small $\alpha$. 
\begin{proposition}
    \label{prop:example} The tests $\process{\tau_\alpha}$ defined in~\eqref{eq:example-test} are optimal~(in the sense of~\Cref{def:first-order-optimality}) for the problem with $(\calP_0, \calP_1)$ as described in~\Cref{def:nonparametric-example}. 
\end{proposition}
The proof of this result is in~\Cref{proof:example}. 
As the proof of~\Cref{prop:example} illustrates, the particular choice of a H\"older ball for the null density class is mainly a convenient way to obtain the structural properties required by our general theory. In particular, the uniform H\"older condition, together with the uniform bounds $0< c_0 \leq \theta \leq C_0$ implies that the family $\Theta_0$ of null densities is uniformly bounded and equicontinuous on the compact domain $\calX = [0, 1]^d$. By the Arzel\'a-Ascoli theorem~(see~\Cref{fact:arzela-ascoli} in~\Cref{apendix:background}), $\Theta_0$ is therefore a compact subset of $\big(\calC(\calX, \R), \unorm\big)$  and a simple continuity argument then shows that $\calP_0 = \{P_{\theta}: \theta \in \Theta_0\}$ is convex and weakly compact. This verifies the condition~\textbf{(A1)} in~\Cref{assump:general}, which is needed to apply the DV variational representation and Sion's minimax theorem in the proof of~\Cref{theorem:general-DV-e-process}. 

On the alternative side, the assumption that each $P \in \calP_1$ is absolutely continuous with respect to some $P_\theta \in \calP_0$, with a bounded continuous likelihood ratio ensures that the optimal DV witness function exists and is a bounded continuous function. The universality of the kernel $\kappa$ then implies that the RKHS $\calH_{\kappa}$ is dense~(in $\sup$ norm) in the space of continuous  functions on $\calX$, which leads to a justification of \textbf{(A2)} in~\Cref{assump:general}. The remaining three conditions of~\Cref{assump:general}, namely \textbf{(A3), (A4)} and \textbf{(A5)} are straightforward consequences of the boundedness and learnability properties of RKHS balls. In this sense, the specific H\"older structure of the null is not essential for our result, and the fundamental requirements are (i) compactness in the appropriate topology of the null class, and (ii) the existence of a bounded continuous optimal log-likelihood ratio that can be well approximated by $\calF$.

\subsection{Practical Aspects}
\label{subsec:computational-aspects}
Our proposed construction of $W_n$ in~\Cref{def:DV-e-process} relies on the ability to solve two optimization problems in each round: 
\begin{itemize}
    \item \sloppy Finding the test function $f_n$ via the empirical risk minimization optimization task: $f_n \in \argmax_{f \in \calF}   \frac 1{n-1} \sum_{i=1}^{n-1} f(X_i) - \psi_0(f)$. 
    \item Evaluating (or upper-bounding) the quantity $\psi_0(f_n) = \sup_{\theta \in \Theta_0} \log \EE_{\theta}[e^{f_n(X)}]$, used for defining the process $\process{W_n}$ in~\Cref{def:DV-e-process}. 
\end{itemize}
In many problem instances, these operations can be computationally prohibitive. In this section, we briefly discuss this issue in two regimes. The first~(\S~\ref{subsubsec:finite-dim}) is a finite-dimensional setting where we may employ existing stochastic saddle-point methods to construct valid test supermartingales, and the second~(\S~\ref{subsubsec:heuristic}) is an infinite dimensional setting where we discuss heuristic approaches based on large-scale generative models. 

\subsubsection{When Both $\Theta_0$ and $\calF$ are Finite Dimensional}
\label{subsubsec:finite-dim}

We first consider the special case in which both the null  and the ``test function class''  admit finite dimensional and convex parameterizations, that is, $\calP_0  = \{P_\theta: \theta \in \Theta_0\}$ and $\calF  = \{f_\varphi: \varphi \in \Phi\}$ for some compact convex sets $\Theta_0 \subset \R^{d_0}$ and $\Phi \subset \R^{d_1}$. The definition of $f_n$ then becomes 
\begin{align}
    f_n \in \argmax_{\varphi \in \Phi} \min_{\theta \in \Theta_0} \calJ_n(\varphi, \theta) \coloneqq \frac{1}{n-1} \sum_{i=1}^{n-1} f_\varphi(X_i) - \log \EE_{\theta}[e^{f_\varphi(X)}]. \label{eq:saddle-point-problem}
\end{align}
In practice, this task can be performed using a stochastic first-order ascent-descent scheme from the stochastic saddle-point literature, such as~\cite[\S~3]{nemirovski2009robust}. Rather than studying a specific algorithm instantiation and its complexity analysis, we formalize the behavior of such algorithms through an \emph{approximate optimization condition}: for some deterministic tolerance level $s_n \downarrow 0$, we assume that for all $n \geq 2$, we can construct an $\calM_{n-1}$-measurable test function $f_n = f_{\hat{\varphi}_n}$ with $\hat{\varphi}_n \in \Phi$, such that we have 
\begin{align}
    \frac{1}{n-1} \sum_{i=1}^{n-1} f_n(X_i)  - \psi_0(f_n)  = \inf_{\theta\in \Theta_0} \calJ_n(\hat{\varphi}_n, \theta)\geq \sup_{\varphi \in \Phi} \inf_{\theta \in \Theta_0} \calJ_n(\varphi, \theta) - s_n, \quad \text{w.p. } \geq 1 - \epsilon_n, \label{eq:approx-test-func}
\end{align}
for some appropriately selected confidence level $\epsilon_n \in (0, 1)$.  
Here $s_n$ may capture the upper bound on the duality gap ensured by the algorithm, and any other Monte Carlo approximations used in the implementation. Having obtained an $f_n$ by the saddle-point solver~(we can set $f_1$ to be an arbitrary element of $\calF$ for $n=1$), the next task is to get an upper bound on $\psi_0(f_n)$. Assuming sample-access to the null model class, we may estimate the map $\theta \mapsto \log E_{\theta}[e^{f_n(X)}]$ with sufficient accuracy, and hence employ off-the-shelf convex solvers to obtain a random quantity $\hat{\psi}_0(f_n)$ satisfying the high probability bound 
\begin{align}
    \PP\lp \hat{\psi}_0(f_n) \geq \psi_0(\hat{f}_n) \mid \calM_{n-1} \rp \geq 1- \eta_n, \label{eq:psi-hat}
\end{align}
for some small $\eta_n \in (0, 1)$, and $n \geq 1$. 
Now,  plugging this $(f_n, \hat{\psi}_0(f_n))$ pair in~\Cref{def:DV-e-process} directly will not give us a valid $e$-process for $\calP_0$ since the required condition may be violated in the $\epsilon_n$-probability ``bad'' event. To address that, we will need to introduce an appropriate correction. We state one such formulation below under a boundedness assumption. 
\begin{proposition}
    \label{prop:finite-dimensional} Suppose there exist finite known constants $0<B_n<\infty$ for all $n \geq 1$, such that 
    \begin{align}
        0 \leq \exp \lp f_n(X_n) - \hat{\psi}_0(f_n) \rp \leq B_n \quad \text{almost surely}, 
    \end{align}
    where $(f_n, \hat{\psi}_0(f_n))$ were introduced in~\eqref{eq:approx-test-func} and~\eqref{eq:psi-hat}. Then, introduce the process 
    \begin{align}
        \Wtilde_0 = 1, \quad \Wtilde_n = \Wtilde_{n-1} \times \Etilde_n \text{ for } n\geq 1, \qtext{where} \Etilde_n = \frac{\exp\big(f_n(X_n) - \hat{\psi}_0(f_n)\big)}{1 + \eta_nB_n}. 
    \end{align}
    Then the process $\process{\Wtilde_n} \coloneqq \{\Wtilde_n: n \geq 0\}$ is a test supermartingale under every $Q \in \calP_0$. Consequently, $\tau_\alpha = \inf \{n \geq 1: \Wtilde_n \geq 1/\alpha\}$ is a level-$\alpha$ test for $\calP_0$ for all $\alpha \in (0, 1)$. 
\end{proposition}
\begin{proof}
    Fix a null distribution $P_\theta \in \calP_0$~(that is, $\theta \in \Theta_0$), and an $n \geq 1$, and let   $\calM_{n-1}$ also include any additional randomness used in the optimization algorithms for finding $(f_n, \hat{\psi}_0(f_n))$. Then, if $G_n$ denotes the ``good event'' $\{\hat{\psi}_0(f_n) \geq \psi_0(f_n)\}$, we know that $\PP_{\theta}(G_n^c \mid \calM_{n-1}) \leq \eta_n$. Then, observe that 
    \begin{align}
         \EE_\theta\lb e^{f_n(X_n) - \hat{\psi}_0(f_n)}  \boldsymbol{1}_{G_n}\mid \calM_{n-1}\rb  &\leq \EE_\theta\lb e^{f_n(X_n) - {\psi}_0(f_n)}  \boldsymbol{1}_{G_n}\mid \calM_{n-1}\rb && (\text{Definition of $G_n$})\\
         & \leq e^{-\psi_0(f_n)} \EE_\theta[e^{f_n(X_n)} \mid \calM_{n-1}] && (f_n \text{ is } \calM_{n-1}\text{-measurable})\\
         &  \leq 1. && (\text{Definition of $\psi_0(f_n)$})
    \end{align}
    Under the bad event, we use the boundedness assumption of~\Cref{prop:finite-dimensional} to get 
    \begin{align}
        \EE_\theta\lb e^{f_n(X_n) - \hat{\psi}_0(f_n)}  \boldsymbol{1}_{G_n^c}\mid \calM_{n-1}\rb  & \leq B_n \PP_\theta(G_n^c \mid \calM_{n-1}) \leq B_n \eta_n. 
    \end{align}
    Combining the previous two displays, we get 
    \begin{align}
        \EE_\theta \lb e^{f_n(X_n) - \hat{\psi}_0(f_n)} \mid \calM_{n-1} \rb \leq 1 + B_n \eta_n,  \quad \implies \quad \EE_\theta [\Etilde_n \mid \calM_{n-1}] \leq 1, 
    \end{align}
    as required for any $\theta \in \Theta_0$. 
\end{proof}

\begin{remark}
    \label{remark:finite-dimension-1} While~\Cref{prop:finite-dimensional} establishes the validity of the test based on $\process{\Wtilde_n}$, for establishing the optimality, in addition to \Cref{assump:general} a sufficient condition is that  $s_n \to 0$, $\sum_{n \geq 1} \epsilon_n < \infty$, and $\sum_{n =1}^{\infty} \eta_n B_n < \infty$. The summability of $\epsilon_n$ ensures that our procedure uses ``bad'' $f_n$ only finitely often, while $s_n \to 0$  in addition to the summability of $B_n \eta_n$ ensures that the long term drift of the process $\process{\log \Wtilde_n}$ is equal to $\gamma^*$. In particular, $\sum_{n \geq 1} \eta_n B_n < \infty $ implies that $\sum_{n \geq 1} \log (1 + \eta_n B_n) < \infty$, and so the additional correction terms in the definition of $\Wtilde_n$ only contribute an $\calO(1)$ term to the rejection (log-) boundary. 
\end{remark}

\subsubsection{Heuristic for Larger Problems}
\label{subsubsec:heuristic}

To tackle the more practically relevant case when neither $\calP_0$ nor $\calF$ admit finite-dimensional parametrization, we may appeal to the recent advances in generative modeling in machine learning, such as diffusion models, invertible neural networks, or generative adversarial networks~(GANs). In fact, the variational structure of the problem~\eqref{eq:saddle-point-problem} closely resembles the training objective used in $f$-GANs.

A concrete practical strategy could be to replace the (possibly infinite dimensional) null class $\calP_0$ with a  rich generative family $\{Q_\eta: \eta \in \calN \}$. Here, each $Q_\eta$ is implemented as a generator that maps some base distribution~(like $Z \sim N(0, I_{d_Z})$ for some $d_Z \geq 1$) into synthetic samples that approximate draws from a null distribution. Similarly, the test-function class $\calF$ may be approximated by another family of neural networks $\{f_\varphi: \varphi \in \Phi\}$ with an appropriate architecture. This leads to the saddle-point problem in which the test-function network is trained to separate the observed data from the null samples, and the null generative-model $\eta$ is trained to generate samples that are hard for the test-function network to distinguish. This exactly mirrors the GAN training pipeline, and we may use alternative stochastic gradient ascent-descent algorithms using appropriate mini-batches of data. 
After sufficiently many iterations of this training, the learned critic function $f_{\hat{\phi}_n}$ and the null generator $Q_{\hat{\eta}_n}$ provide an approximate DV optimizer pair, and we can define a heuristic  $e$-process by setting $f_n = f_{\hat{\phi}_n}$, and obtain $\hat{\psi}_0(f_n)$  via a Monte Carlo approximation using the generator $Q_{\hat{\eta}_n}$. 

Due to the nonconvexity/nonconcavity of the objective, the above approach does not yield  a critic function  $f_n = f_{\hat{\phi}_n}$ satisfying the conditions required stated in~\Cref{assump:general}, nor is $\hat{\psi}_0(f_n)$ defined above guaranteed to be an upper bound on $\psi_0(f_n)$. Nevertheless, this GAN-style construction is a natural and flexible heuristic to implement our general ideas to practical and high-dimensional problems, and we believe that exploring such implementations empirically as a promising direction for future work.

\section{Conclusion}
\label{sec:conclusion}
In this paper, we considered the problem of constructing optimal level-$\alpha$ power-one tests for weakly compact and composite null hypotheses. Prior work has established a fundamental lower bound in the limit of small $\alpha$ of the expected value of any such test in terms of the minimum relative entropy to the null~(or $\mathrm{KL}_{\inf})$. That is, for any collection of tests $\process{\tau_\alpha} = \{\tau_\alpha: \alpha \in (0, 1)\}$ for a given  null class $\calP_0$, and any $P \not \in \calP_0$, we must have  the following:
\begin{align}
\liminf_{\alpha \downarrow 0} \frac{\EE_P[\tau_\alpha]    }{J(P, \calP_0, \alpha)} \geq 1, \qtext{where} J(P, \calP_0, \alpha) = \frac{\log(1/\alpha)}{\inf_{Q \in \calP_0} \dkl(P \parallel Q)}. 
\end{align}
\sloppy In this paper, we established the existence of tests that are optimal in the sense that they satisfy $\limsup_{\alpha \downarrow 0} \EE_P[\tau_\alpha]/J(P, \calP_0, \alpha)\leq 1$. Towards this goal, we first considered the case of distributions supported on finite alphabets, and proved that a sequential test based on a class of $e$-processes proposed by~\citet{wasserman2020universal} are optimal in the above-described sense. A crucial step in proving this result was a saddle-point representation of $\gamma^*(P, \calP_0)$, and an interchange of the order of $\inf$ and $\sup$ to obtain an ``oracle $e$-process''. Using this as a starting point, we then considered a more general class of problems for which we can incrementally approximate the oracle $e$-process in a data-driven manner, and establish the optimality of the associated power-one test. To illustrate the power of this result, we discussed a detailed example of a nonparametric null parametrized by a H\"older ball of densities, and an alternative class for which the likelihood ratio is a continuous bounded function that can be approximated by elements of a reproducing kernel Hilbert space. Finally, we also discussed some computational aspects associated with implementing our proposed saddle-point based tests. 

The results of this paper present several interesting directions for future work. Building on the discussion of~\Cref{subsec:computational-aspects}, an immediate direction for investigation is to employ iterative algorithms from the literature on stochastic average approximation algorithms to design provably optimal tests for finite dimensional settings. On a more practical note, exploring how modern generative models may be used to construct tests for large-scale problems, such as by incorporating constraints associated with null classes is an important topic. Finally, extending our ideas beyond the \iid case is another important direction.

\bibliographystyle{abbrvnat}
\bibliography{ref}

\begin{appendix}
\section{Additional Background}
\label{apendix:background}
In this section, we recall some key facts that serve as the foundations of the theoretical results presented in this paper. We begin by recalling a classical variational representation of relative entropy derived by~\citet{donskervaradhan1983asymptotic}. We state a version of this result following the notation of~\citet{polyanskiy2025information}. 
\begin{fact}[Donsker-Varadhan~(DV) Representation]
    \label{fact:donsker-varadhan} Let $P, Q$ denote two probability measures on some measurable space $(\calX, \calB_\calX)$, such that $P \ll Q$. Then, the relative entropy between $P$ and $Q$ is defined as $\EE_P\lb \log\tfrac{dP}{dQ(X)} \rb$, and satisfies the following variational representation: 
    \begin{align}
        \dkl(P \parallel Q) = \sup_{f \in \calC_Q} \EE_P[f(X)] - \log \EE_Q[e^{f(X)}], 
    \end{align}
    where $\calC_Q = \{f: \calX \to \R: 0< \EE_Q[e^{f(X)}] < \infty\}$. Furthermore,  if $P \ll Q$ and $\dkl(P \parallel Q) < \infty$,  the supremum in the expression above is attained at any $f^* \equiv f^*(P, Q)$ of the form $f^*(x) = \log \tfrac{dP}{dQ}(x) + c$, for all $x \in \calX$ and any constant $c \in \mathbb{R}$.   If $(\calX, \calB)$ is a metric space with the Borel sigma-algebra, then we can replace $\calC_Q$ with $\calC_b$, the class of all bounded continuous functions on $\calX$. 
\end{fact}

Next, we recall a general minimax theorem first obtained by~\citet[Corollary  3.3]{sion1958general}, whose proof has subsequently been simplified, for example by~\citet{komiya1988elementary}. 
\begin{fact}[Sion's Minimax Theorem]
    \label{fact:Sions-Minimax-Theorem} 
    Let $\mathbb{X}$ be a compact convex subset of a topological vector space $\calX$, and $\mathbb{Y}$ be a convex subset of a topological vector space $\calY$. Let $\Phi:\mathbb{X} \times \mathbb{Y} \to \mathbb{R}$ be a function such that 
    \begin{itemize}
        \item For every $x \in \mathbb{X}$, the map $y \mapsto \Phi(x, y)$ is quasi-convex and lower semicontinuous on $\mathbb{Y}$. 
        \item For every $y \in \mathbb{Y}$, the map $x \mapsto \Phi(x, y)$ is quasi-concave and upper semicontinuous on $\mathbb{X}$. 
    \end{itemize}
    Then, we have the following: 
    \begin{align}
        \inf_{x \in \mathbb{X}}\;\sup_{y \in \mathbb{Y}} \Phi(x, y) \; = \; 
        \sup_{y \in \mathbb{Y}} \; \inf_{x \in \mathbb{X}} \Phi(x, y). 
    \end{align}
\end{fact}
We now recall a result~\citep[Theorem C, Pg. 126]{simmons1963topology} that allows us to establish the compactness condition needed to apply the minimax theorem in our example discussed in~\Cref{subsec:example}. 
\begin{fact}[Arzel\'a-Ascoli]
    \label{fact:arzela-ascoli} Let $(\mathbb{X}, \rho)$ denote a compact metric space, and let $\Theta$ denote a closed subset of $\big(C(\mathbb{X}, \mathbb{R}), \unorm\big)$, the space of continuous functions on $\mathbb{X}$ endowed with the $\sup$-norm $\|f\|_\infty = \sup_{x \in \mathbb{X}}|f(x)|$. Then, $\Theta$ is compact if and only if it is bounded and equicontinuous. 
\end{fact}

Finally, we recall a simplified version of a martingale inequality derived by~\citet{pinelis1994optimum} that is used in the proof of~\Cref{theorem:general-DV-e-process}. 
\begin{fact}[adapted from Theorem~4.1 of~\citet{pinelis1994optimum}]
    \label{fact:pinelis-1994}
    Suppose $\{Y_i: i \geq 1\}$ denotes a real-valued martingale difference sequence with respect to some filtration $\{\calM_i: i \geq 0\}$. Define the partial sums,  $S_k = \sum_{i=1}^k Y_i$ for any $k \in [n]$,  and the maximum values $S^*_n = \max_{1 \leq k \leq n} |S_k|$ and $Y^*_n = \max_{1 \leq i \leq n} |Y_i|$. Then, for every $p \geq 2$, there exist finite positive constants $C_{1,p}, C_{2,p}$~(independent of $n$), such that 
    \begin{align}
        \EE \lp \lb |S^*_n|^p \rb \rp^{1/p} \leq C_{1,p} \lp \EE[(Y^*_n)^p] \rp^{1/p} + C_{2,p} \lp \EE\lb \lp\sum_{i=1}^n \EE\lb Y_i^2 \mid \calM_{i-1} \rb  \rp^{p/2} \rb \rp^{1/p}. 
    \end{align}
\end{fact}

\section{DV $e$-process Instantiations}
\label{appendix:DV-instantiations}
In this section, we instantiate the general $e$-process introduced in~\Cref{def:DV-e-process} to two important settings: (i) finite alphabet with composite null, and (ii) testing the mean of bounded observations. 

\subsection{Finite Alphabet Setting}
\label{appendix:instantiation-finite-alphabet}

We return to the case of testing with observations lying in some finite alphabet $\calX = \{x_1, \ldots, x_m\}$, and a null class $\calP_0 = \{P_\theta: \theta \in \Theta_0\}$ for some compact convex subset $\Theta_0$ of the $(m-1)$-dimensional simplex~(here $P_\theta$ denotes the distribution on $\calX$ with pmf $\theta$). Then, let us introduce the class of functions: 
\begin{align}
    \calF = \{\log{\varphi}: \varphi \in \Phi\}, \qtext{with} \Phi \coloneqq [\epsilon, 1/\epsilon]^m. 
\end{align}
This $\calF$ corresponds to all alternative distributions with pmf $\theta$, such that there exists a $\theta_0 \in \Theta_0$ with $\theta/\theta_0 \in \Phi$. Then, the function $f_n$ in round $n \geq 2$ is defined as 
\begin{align}
    f_n = \log \varphi_n, \qtext{with} \varphi_n \in \argmax_{\varphi \in \Phi} \frac{1}{n-1} \sum_{i=1}^{n-1} \log \varphi[X_i] - \max_{\theta_0 \in \Theta_0} \log \sum_{x \in \calX} \theta_0[x] \varphi[x]. 
\end{align}
It is easy to verify that this is a finite-dimensional convex-concave saddle point problem and can be solved with sufficient accuracy using existing methods. Having obtained the parameter $\varphi_n$~(or function $f_n$), we can then approximate $\psi_0(f_n)$, by solving 
\begin{align}
    \max_{\theta_0 \in \Theta_0} \log \sum_{x \in \calX} \theta_0[x] \varphi_n[x]. 
\end{align}
This is again a convex program over a compact convex domain, and can be solved using off-the-shelf solvers with high accuracy. With these two terms at hand, we can then construct the e-process, perhaps with some additional correction if necessary, as discussed in~\Cref{subsubsec:finite-dim}. 

\subsection{Bounded Mean Testing}
\label{appendix:instantiation-bounded-mean}
We now discuss the application of our DV $e$-process for testing the mean of a stream of bounded observations. In particular, assume that $\calX = [0,1]$, and let 
\begin{align}
    \calP_0 = \{P \in \calP(\calX): \EE_P[X] = \mu_0\}, \qtext{for some known} \mu_0 \in \calX = [0,1].  \label{eq:bounded-mean-null-class}
\end{align}
In other words, the null  class consists of all distributions supported on $\calX$ with mean exactly equal to $\mu_0$. Let us introduce the test-function class for some small $\epsilon \in (0, 1)$: 
\begin{align}
    \calF = \lbr x \mapsto  \log (1 + \varphi(x- \mu_0)): \varphi  \in \Phi_\epsilon\rbr, \qtext{with} \Phi_\epsilon = \lb -\frac{1-\epsilon}{1-\mu_0}, \frac{1-\epsilon}{\mu_0} \rb. \label{eq:bounded-mean-test-functions}
\end{align}
The specific choice of this particular test function class is motivated by two reasons. The first is that for any $f \equiv f_{\varphi} \in \calF$, we have 
\begin{align}
    \psi_0(f_\varphi) = \sup_{Q \in \calP_0} \log \lp \EE_Q[\exp \lp \log(1 + \varphi(X-\mu_0)) \rp] \rp= \sup_{Q \in \calP_0} \log \lp 1 + \varphi\lp EE_Q[X] - \mu_0 \rp \rp = 0. \label{eq:bounded-mean-psi-0}
\end{align}
The second reason is that a duality argument by~\citet{honda2010asymptotically} tells us that for any $P \not \in \calP_0$, we have the following representation of the $KL_{\inf}$ term 
\begin{align}
    \gamma^*(P, \calP_0) = \sup_{\varphi \in \lb -\frac{1}{1-\mu_0}, \frac{1}{\mu_0}\rb } \EE_P[\log(1 + \varphi(X-\mu_0))]. 
\end{align}
Together, these two results imply that our choice of $\calF$~(with $\epsilon >0$) can be used to design optimal tests against alternatives for which $\varphi^* \equiv \varphi^*(P, \calP_0)$ lies in the $\epsilon$-interior of the domain in the above dual representation. 
\begin{align}
    \calP_1 = \{P \in \calP(\calX): \varphi^*(P, \calP_0) \in \Phi_\epsilon\}, \qtext{which implies}
    d_{\calF}(P \parallel \calP_0) = \gamma^*(P, \calP_0), \label{eq:bounded-mean-alt-class}
\end{align}
for all $P \in \calP_1$. We can now state the main result of this section. 

\begin{proposition}
    \label{prop:bounded-mean} Given a stream of $\calX = [0,1]$-valued observations $\{X_n: n \geq 1\}$ drawn \iid from a distribution $P \in \calP_0 \cup \calP_1$, with the classes $\calP_0$ and $\calP_1$ introduced above in~\eqref{eq:bounded-mean-null-class} and~\eqref{eq:bounded-mean-alt-class} respectively. Then, define the process $\process{W_n}$, with $W_0 = W_1= 1$ and 
    \begin{align}
        W_n = W_{n-1} \times \lp 1 + \varphi_n (X_n - \mu_0) \rp, \qtext{with} \varphi_n \in \argmax_{\varphi \in \Phi_\epsilon} \frac{1}{n-1} \sum_{i=1}^{n-1} \log \lp 1 + \varphi(X_i - \mu_0) \rp, 
    \end{align}
    for all $n \geq 2$. Then, with $\tau_\alpha = \inf \{n \geq 1: W_n \geq 1/\alpha\}$ for $\alpha \in (0,1)$, we have 
    \begin{align}
        \limsup_{\alpha \downarrow 0} \frac{\EE_P[\tau_\alpha]}{\log(1/\alpha)/\gamma^*(P, \calP_0)} \leq 1, \qtext{for all} P \in \calP_1. 
    \end{align}
    In other words, the level-$\alpha$ tests $\process{\tau_\alpha}$ are optimal in the sense of~\Cref{def:first-order-optimality}. 
\end{proposition}
\begin{proof}
This result follows by simply verifying the conditions of~\Cref{assump:general}, and then appealing to~\Cref{theorem:general-DV-e-process}. Since all the distributions are supported on $[0,1]$, an application of Prohorov's theorem can be used to verify the compactness of $\calP_0$. The convexity of $\calP_0$ is a simple consequence of the linearity of expectation. Together, these two observations show that the condition \textbf{(A1)} is satisfied by $\calP_0$. Condition~\textbf{(A2)} is ensured by definition of $\calP_1$, and \textbf{(A4)} is also trivially satisfied since $\psi_0(f)=0$ for all $f \in \calF$ as we saw in~\eqref{eq:bounded-mean-psi-0}. Finally, standard covering-number arguments can be used to show the learnability of the function class $\calF$ considered, thus verifying~\textbf{(A5)}.
\end{proof}

\begin{remark}
    \label{remark:bounded-mean} Note that this construction of $\process{W_n}$ and $\process{\tau_\alpha}$  discussed above is exactly the ``betting-based'' test with the so-called GRAPA betting strategy developed by~\citet{waudby2024estimating}. The term $\varphi_n$ can be computed using off-the-shelf convex optimization solvers in a computationally feasible manner. Thus, our result shows that the GRAPA-based test is optimal in the small-$\alpha$ regime for alternatives whose optimal ``bet'' $\varphi^*(P, \calP_0)$ lies in the $\epsilon$-interior set $\Phi_\epsilon$. This restriction can be removed by simply using a countable mixture over function classes, that we introduced earlier in~\Cref{subsec:example}, and we omit the details to avoid repetition. 
\end{remark}

\section{Proof of~\Cref{theorem:UI-test-finite} (Finite Alphabet)} 
\label{proof:UI-test}

The level-$\alpha$ property of this test has been established by~\citet{wasserman2020universal}. We include the details here for completeness.  The property essentially hinges on the observation that under $H_0$, if $\theta_0$ is the true parameter, then 
\begin{align}
    \UIevalue_n = \frac{\theta_J^n[X^n]}{\thetahat^n_0[X^n]} = \frac{\theta_J^n[X^n]}{\sup_{\theta \in \Theta_0} \theta^n[X^n]} \; \stackrel{a.s.}{\leq} \; \frac{\theta_J^n[X^n]}{ \theta_0^n[X^n]} \eqcolon L_n^{\theta_0}. 
\end{align}
It is easy to verify that under $\theta_0$, the process $\{L_n^{\theta_0}: n \geq 0\}$ is a nonnegative martingale  with an initial value of $L_0^{\theta_0}=1$, which implies that 
\begin{align}
    \mathbb{P}_{\theta_0}(\tau_\alpha < \infty) = \mathbb{P}_{\theta_0}\lp \exists n \geq 1: \UIevalue_n \geq 1/\alpha \rp \leq   \mathbb{P}_{\theta_0}\lp \exists n \geq 1: L_n^{\theta_0} \geq 1/\alpha \rp \leq \alpha, 
\end{align}
where the last inequality follows from an application of Ville's inequality.  Taking a supremum over all $\theta_0 \in \Theta_0$ gives us the first part of~\Cref{theorem:UI-test-finite}. 

\paragraph{Expected Stopping Time.} To study the expected stopping time property, consider the alternative with the true parameter $\theta_1 \in \Theta_1$. By the regret of the universal compression scheme, we know that 
\begin{align}
    S_n \coloneqq \log \UIevalue_n \geq  \sup_{\theta \in \Theta} \log \theta^n[X^n] - \log {\thetahat^n_0[X^n]} - \Reg_n, \qtext{with} \Reg_n = \frac{m-1}{2}\log n + C_m, 
\end{align}
for some constant $C_m$ independent of $\theta_1$. 
Let $\phat_n$ denote the empirical probability mass function~(pmf) based on $X^n$. Observe that for any $\theta \in \Theta$, we have 
\begin{align}
    \log \theta^n[X^n] = n \sum_{x \in \calX} \phat_n[x] \log \theta[x] = n \lp - H(\phat_n) - \dkl(\phat_n \parallel \theta) \rp. 
\end{align}
Here $H(\cdot)$ denotes the Shannon entropy for discrete distributions, and $\dkl$ is the relative entropy or KL-divergence. 
By the nonnegativity of relative entropy, this implies that $\phat_n$ is also the MLE (over all $\Theta$) based on $X^n$, which means that 
\begin{align}
    S_n \geq - nH(\phat_n) - \log \thetahat^n_0[X^n] - \Reg_n. \label{eq:stopping-time-composite-proof-1}
\end{align}
Now, we analyze the second term above to see that 
\begin{align}
    \log \sup_{\theta_0 \in \Theta_0} \theta_0^n[X^n] &= n \sup_{\theta_0 \in \Theta_0}\sum_{x \in \calX} \phat_n \log \theta_0[x]  = -n  \sup_{\theta_0 \in \Theta_0}  \lp H(\phat_n) + \dkl(\phat_n \parallel \theta_0) \rp \\
    & = -n H(\phat_n) -n \dkl(\phat_n \parallel \Theta_0), \label{eq:stopping-time-composite-proof-2}
\end{align}
where we define $\dkl(\phat_n \parallel \Theta_0)$ as $\inf_{\theta_0 \in \Theta_0} \dkl(\phat_n \parallel \theta_0)$. On combining~\eqref{eq:stopping-time-composite-proof-1} and~\eqref{eq:stopping-time-composite-proof-2}, we get 
\begin{align}
    S_n \geq n \dkl(\phat_n \parallel \Theta_0) - \Reg_n. \label{eq:stopping-time-composite-proof-3}
\end{align}
Now, we recall the Donsker-Varadhan~(DV) variational representation~\citep[Theorem 4.6]{polyanskiy2025information}  of relative entropy, to observe that 
\begin{align}
\dkl(p \parallel \Theta_0) = \inf_{\theta_0 \in \Theta_0} \sup_{f:\calX \to \mathbb{R}} \lp \mathbb{E}_p[f] - \log E_{\theta_0}[e^f] \rp \eqcolon  \inf_{\theta_0 \in \Theta_0} \sup_{f:\calX \to \mathbb{R}} \Phi(f, \theta_0).
\end{align}

Now, observe that the function $\Phi$ is concave in $f$~(due to the convexity of log-sum-exp function), and convex in  $\theta_0$~(due to the concavity of $\log$ function). We have assumed that the null class $\Theta_0$ is compact and convex, and additionally note that the class of functions $\R^{\calX}$ is convex. Hence, the hypotheses of Sion's minimax theorem~(recalled in~Fact~\ref{fact:Sions-Minimax-Theorem}) are satisfied, and we have 
\begin{align}
    &\dkl(p \parallel \Theta_0) = \sup_{f:\calX \to \mathbb{R}} \inf_{\theta_0 \in \Theta_0} \mathbb{E}_p[f] - \log \mathbb{E}_{\theta_0}[e^f] =  \sup_{f:\calX \to \mathbb{R}}  \mathbb{E}_p[f] - \psi_0(f), 
\end{align}
where  $\psi_0(f) \coloneqq \sup_{\theta_0 \in \Theta_0} \log \mathbb{E}_{\theta_0}[e^f]$.  Now, let us fix an arbitrary (for now) $f:\calX \to \mathbb{R}$, which gives a lower bound in the equality above, and using this with~\eqref{eq:stopping-time-composite-proof-3}, we get 
\begin{align}
S_n \geq n \dkl(\phat_n \parallel \Theta_0) \geq \mathbb{E}_{\phat_n}[f] - \psi_0(f) - \Reg_n = \underbrace{\sum_{i=1}^n f(X_i) - n \psi_0(f)}_{\coloneqq S_n(f)} - \Reg_n. 
\end{align}
Thus, the stopping time $T_f = \inf \{n \geq 1: S_n(f) \geq \log(1/\alpha) + \Reg_n\}$ is an upper bound on $\tau_\alpha = \inf \{n \geq 1 : S_n \geq \log(1/\alpha)\}$. As $T_f$ is an upper bound on $\tau_\alpha$ for arbitrary $f$, this relation holds for the $f$ defined as follows: 
\begin{align}
    f \in \argmax_{f':\calX \to \mathbb{R}} \mathbb{E}_{\theta_1}[f'] - \psi_0(f')  \quad \implies \quad \mathbb{E}_{\theta_1}[f] - \psi_0(f) = \dkl(\theta_1 \parallel \Theta_0) = \gamma_{\theta_1}. \label{eq:optimal-f}
\end{align}
Next, let $\theta^\dagger$ denote any minimizer in the definition of $\gamma^*(\theta_1, \Theta_0)$; that is, 
\begin{align}
    \theta^\dagger \in \argmin_{\theta \in \Theta_0} \dkl(\theta_1 \parallel \theta), 
\end{align}
whose existence is justified by the compactness of $\Theta_0$, and the lower-semicontinuity of relative entropy. If $\gamma^*(\theta_1, \Theta_0) < \infty$, then we must have $\theta^{\dagger}[x]>0$ for all $x $ such that $\theta_1[x]>0$. Finally, note that the optimal $f$ defined in~\eqref{eq:optimal-f} is of the form $f(x) = \log \theta_1[x]/\theta^\dagger[x] + c$, and we can set $c=0$ without loss of generality. 
With these choices of $f$ and $\theta^\dagger$, we can conclude the following: 
\begin{itemize}

\item Since $f(x) = \log \theta_1[x]/\theta^\dagger[x]$ and $\theta^\dagger[x]>0$ for all $x$ such that $\theta_1[x]>0$, we can observe that 
\begin{align}
    f(x) - \psi_0(f) \leq f(x) - \sup_{\theta_0\in \Theta_0} \log \sum_{y \in \calX} \theta_0[y]e^{f(y)} \leq f(x) - \log \sum_{y \in \calX} \theta^\dagger[y] \times\frac{\theta_1[y]}{\theta^\dagger[y]} = f(x). 
\end{align}
    Hence, the one-step increment of $S_n(f)$ is upper bounded by $C_f \coloneqq \max_{x \in \calX} \log \theta_1[x]/\theta^\dagger[x]$. 

\item The stopping time $T_f$ has finite expectation. To see this, note that $\Reg_n$ is logarithmic in $n$, so there exists some finite $n_0$, such that for all $n \geq n_0$, we have $\Reg_n + \log(1/\alpha) \leq n \gamma^*(\theta_1, \Theta_0)/2$. Hence, for all $n$ larger than $n_0$, we have 
\begin{align}
    \PP_{\theta_1}( T_f  > n) \leq \PP_{\theta_1}\lp S_n(f) < \log(1/\alpha) + \Reg_n \rp \leq\PP_{\theta_1}\lp S_n(f) - \EE_{\theta_1}[S_n(f)] < - \tfrac 12  \EE_{\theta_1}[S_n(f)] \rp. 
\end{align}
Since the \iid increments of $S_n(f)$ are bounded, we can apply Hoeffding's inequality to obtain 
\begin{align}
    \EE_{\theta_1}[T_f] \leq n_0 + \sum_{n \geq n_0} \PP_{\theta_1}\lp T_f > n \rp \leq n_0 + \sum_{n \geq n_0}e^{-c n} < \infty, 
\end{align}
for some constant $c>0$ depending on the range of increments of $S_n(f)$. 

\end{itemize}
Together, these two facts imply that we can use Wald's identity to conclude that 
\begin{align}
   \mathbb{E}_{\theta_1}[S_{T_f}(f)] =  \mathbb{E}_{\theta_1}[T_f] \mathbb{E}_{\theta_1}[f(X) - \psi_0(f)] \leq \log(1/\alpha) + \frac{m-1}{2} \log \mathbb{E}_{\theta_1}[T_f] + C_m + C_f. 
\end{align}
Let $\gamma^*(\theta_1, \Theta_0)$ denote $\mathbb{E}_{\theta_1}[f(X) - \psi_0(f)]$, and observe that with $y = \mathbb{E}_{\theta_1}[T_f]$, we have 
\begin{align}
     y \gamma^*(\theta_1, \Theta_0) \leq \log(1/\alpha) + \frac{m-1}{2} \log y + C_m + C_f. 
\end{align}
By some standard calculations~(see~\Cref{lemma:fixed-point-1} for details) we can show that 
\begin{align}
    \mathbb{E}_{\theta_1}[T_f] 
    &= \frac{\log(1/\alpha)}{\gamma^*(\theta_1, \Theta_0)} \lp 1 + o(1)\rp, 
\end{align}
where $o(1)$ indicates terms that go to zero as $\alpha \to 0$ for a fixed $\theta_1$. This completes our proof. \hfill \qedsymbol

\section{Proof of~\Cref{theorem:general-DV-e-process}~(General Alphabet)}
\label{proof:general-DV-e-process}
 To start with, observe that due to the weak compactness and convexity of $\calP_0$ assumed in~\textbf{(A1)}, along with the (weak) lower semicontinuity of relative entropy, for any $P \not \in \calP_0$, we have 
\begin{align}
\gamma^* \equiv \gamma^*(P, \calP_0) \coloneqq \inf_{P_0 \in \calP_0} \dkl(P \parallel P_0) = \dkl(P \parallel P^*_0),     
\end{align}
for some element $P^*_0 \equiv P^*_0(P) \in \calP_0$. Now, using the Donsker-Varadhan representation~(Fact~\ref{fact:donsker-varadhan}), we have 
\begin{align}
    \gamma^* = \inf_{P_0 \in \calP_0} \sup_{f \in \calC_b(\calX)} \EE_P[f(X)] - \log \EE_{P_0}[e^{f(X)}] =  \sup_{f \in \calC_b(\calX)} \inf_{P_0 \in \calP_0} \EE_P[f(X)] - \log \EE_{P_0}[e^{f(X)}], 
\end{align}
where $\calC_b(\calX)$ denotes the space of continuous bounded functions on the metric space $(\calX, \rho)$, and the swap of the $\inf$-$\sup$  in the second equality is justified by the  minimax theorem recalled in~Fact~\ref{fact:Sions-Minimax-Theorem}. Finally, note that by condition~\textbf{(A2)} of~\Cref{assump:general}, we know that we can restrict our attention to the class $\calF$ without loss of optimality in $\gamma^*$; that is, 
\begin{align}
    \gamma^*(P, \calP_0) = \sup_{f \in \calF} \EE_P[f(X)] - \psi_0(f), \qtext{where} \psi_0(f) = \sup_{P_0 \in \calP_0} \log \EE_{P_0}[e^{f(X)}]. 
\end{align}
Hence, for any arbitrary $\varepsilon \in (0, 1/3)$, there exists an $f^* \in \calF$, such that $\EE_P[f^*(X)] - \psi_0(f^*) \geq \gamma^*(1 - \varepsilon)$.  

For the rest of the proof, we will use $C$ to denote any universal positive constant, whose exact value may change at different points.
For any $i \geq 1$, introduce the terms
\begin{align}
    V_i = f_i(X_i) - \psi_0(f_i), \qtext{and} \Vbar_i = V_i - \EE[V_i \mid \calM_{i-1}], \qtext{for all} i \geq 1. 
\end{align}
With these terms, we define the following events: 
\begin{align}
    G_{n,1} = \lbr \EE[ V_n \mid \calM_{n-1}] \geq \gamma^*(1- \varepsilon) - 2 r_n \rbr, \qtext{and}
    G_{n,2} = \lbr \sum_{i=1}^n \Vbar_i  \geq -C n^a \rbr, \label{eq:Gn1-Gn2-def}
\end{align}
where $a = 1 - \delta/(6(4+\delta)) \in (0,1)$. Both $G_{n,1}$ and $G_{n,2}$ occur with probability at least $1-1/n^{2+\delta/3}$ each, as we show in~\Cref{lemma:uniform-erm} for $G_{n,1}$ and~\Cref{lemma:martingale-concentration} for $G_{n,2}$. In particular,~\Cref{lemma:uniform-erm} relies on the uniform convergence condition~\textbf{(A5)}, while~\Cref{lemma:martingale-concentration} uses the uniform $4+\delta$ condition~\textbf{(A3)}, and the uniform log-MGF bound assumed in~\textbf{(A4)}. 
As a result, we have the following, with $H_n \coloneqq \cap_{m \geq n} G_m$: 
\begin{align}
    \sum_{n \geq 1} \PP(H_n^c) = \sum_{n \geq 1} \PP\big( \lp \cap_{m \geq n} G_m \rp^c \big) \leq \sum_{n \geq 1} \sum_{m \geq n} \frac{2}{m^{2 + \tfrac{\delta}{3}}} < \infty. \label{eq:event-Hn-def}
\end{align}
Hence, there exists an almost surely finite random variable $N^*$ such that the event $G_n$ occurs for all $n \geq N^*$. In other words, $\{N^* \geq n\} \subset \cup_{m \geq n} G_m^c$, which implies that 
\begin{align}
    \EE[N^*] &= \sum_{n \geq 1} \PP(N^* \geq n) \leq \sum_{n \geq 1} \sum_{m \geq n} \PP(G_m^c) \leq \sum_{n \geq 1} \sum_{m \geq n} \frac{2}{m^{2 + \tfrac{\delta}{3}}} \eqcolon C^* \equiv C^*(\delta) <  \infty. \label{eq:Nstar-bound}
\end{align}
Now, for the same value of  $\varepsilon \in (0, 1/3)$, define 
\begin{align}
    N_0 = \inf \{n \geq 1: 2r_n \leq \varepsilon \gamma^*(1-\varepsilon) \}, \qtext{and} N_1 = \max \{N_0, \, N^*\}. \label{eq:N1-def} 
\end{align}
Since we have assumed that $r_n \to 0$ with $n$, it means that for any $\varepsilon>0$, the term $N_0$ is a finite deterministic constant depending on $\varepsilon$,  and $\gamma^*$.  For any $n \geq 1$, we have the following, with $\gamma_\varepsilon \coloneqq \gamma^*(1-2 \varepsilon) \leq \gamma^*(1-\varepsilon)^2 $ and some constant $B' < \infty$~(depending on $B_P$ and $B_0$ from~\Cref{assump:general}): 
\begin{align}
    \log W_n & = \sum_{i=1}^n V_i = \sum_{i=1}^n \Vbar_i + \sum_{i=1}^n \EE[V_i \mid \calM_{i-1}] && (\text{since } \Vbar_i = V_i - \EE[V_i \mid \calM_{i-1}])\\ 
    & = \sum_{i=1}^n \Vbar_i  + \sum_{i=1}^{N_1-1} \EE[V_i \mid \calM_{i-1}] + \sum_{i=N_1}^n \EE[V_i \mid \calM_{i-1}]  \\ 
    & \geq \sum_{i=1}^n \Vbar_i - N_1 B' +  \sum_{i=N_1}^n \EE[V_i \mid \calM_{i-1}]  && (\text{by assumptions \textbf{(A3)} \& \textbf{(A4)}}) \\ 
    & \geq \sum_{i=1}^n \Vbar_i - N_1 B' + (n-N_1) \gamma_\epsilon. && (\text{by definition of $N_1$ in~\eqref{eq:N1-def}}) \\
    & = \sum_{i=1}^n \Vbar_i + n \gamma_\epsilon - N_1 \underbrace{(B'+\gamma_\epsilon)}_{\coloneqq B_1}.  \label{eq:logWn-lower-bound}
\end{align}

Since $\tau$ is a nonnegative integer valued random variable, we have the following with $H^* = \cup_{n \geq 1} \cap_{m \geq n} G_m$ denoting the probability one event introduced earlier: 
\begin{align}
    \EE[\tau] & = \sum_{n \geq 1} \PP(\tau \geq n) = \sum_{n \geq 1} \EE[\boldsymbol{1}_{\tau \geq n}]  \stackrel{\mathrm{MCT}}{=} \EE\bigg[ \sum_{n \geq 1} \boldsymbol{1}_{\tau \geq n} \bigg] = \EE\bigg[ \sum_{n = 1}^{N_1-1} \boldsymbol{1}_{\tau \geq n} +  \sum_{n \geq N_1} \boldsymbol{1}_{\tau \geq n} \bigg] \\
    & \leq  \EE\bigg[ N_1 +  \sum_{n \geq N_1} \boldsymbol{1}_{\tau \geq n} \bigg] = \EE[N_1] + \EE \bigg[ \sum_{n \geq N_1} \boldsymbol{1}_{\tau \geq n} \bigg] = \EE[N_1] + \EE \bigg[ \sum_{n \geq N_1} \boldsymbol{1}_{\tau \geq n} \boldsymbol{1}_{H^*} \bigg]. \label{eq:tau-general-upper-bound-proof-1}
\end{align}
We now look at the event $\{\tau \geq n\} \cap H^*$ for $n \geq N_1$, with $A$ denoting $\log(1/\alpha)$:  
\begin{align}
    \{\tau \geq n\} \cap H^* & \subset \{ \log W_n< A\} 
     \subset \{ \sum_{i=1}^n \Vbar_i + n \gamma_\varepsilon - N_1 B_1 < A \} \cap H^* && (\text{by equation~\eqref{eq:logWn-lower-bound}}) \\
     & \subset \lbr n \gamma_\varepsilon - N_1 B_1 - C n^a < A\rbr. && (\text{by equation~\eqref{eq:Gn1-Gn2-def}}) \label{eq:tau-general-upper-bound-proof-2}
\end{align}
Now, for any $u>0$, define the term 
\begin{align}
    M(u) = \inf \{n \geq 1: \forall m \geq  n, \; m \gamma_\varepsilon - C m^a \geq A + u \}  \leq \frac{A+u}{\gamma^* (1-3\varepsilon)} + \lp \frac{C}{\gamma^* \varepsilon} \rp^{\tfrac{1}{1-a}}.
\end{align}
The inequality uses the fact that for $m \geq (C/\gamma^* \varepsilon)^{1/1-a}$, we have $m \gamma_{\varepsilon} - C m^a \geq m \gamma_{\varepsilon} - m \gamma^* \varepsilon = \gamma^*(1-3\varepsilon)$. 
Next, we observe that~\eqref{eq:tau-general-upper-bound-proof-1} and~\eqref{eq:tau-general-upper-bound-proof-2} combined with the definition of $M(\cdot)$ above leads to 
\begin{align}
    \EE[\tau] & \leq \EE[N_1] + \EE\big[M\big( N_1 B_1) \big) \big] \leq \EE[N_1] + \frac{B_1  \EE[N_1] + A}{\gamma^*(1-3\varepsilon)} + \lp \frac{C}{\gamma^* \varepsilon} \rp^{\tfrac{1}{1-a}}. \label{eq:tau-general-upper-bound-proof-3}
\end{align}
Since $N_1 = \max \{N_0, N^*\}$, it means that 
\begin{align}
    \EE[N_1] \leq \EE[N^*] + \EE[N_0] \leq C^* + N_0,
\end{align}
where we used the upper bound on $N^*$ obtained in~\eqref{eq:Nstar-bound}, and the  fact that $N_0$ is a deterministic constant depending on $\varepsilon$ and $\gamma^*$.  Thus, using this in~\eqref{eq:tau-general-upper-bound-proof-3}, we get 
\begin{align}
    \EE[\tau] & \leq  \frac{A}{\gamma^*(1-3\varepsilon)}\lbr 1 +  \frac{\EE[N_1]}{A} \lp \frac{B_1}{\gamma^*}  + \gamma^*(1-3\varepsilon) \rp  +  \frac{\gamma^*(1-3\varepsilon)}{A}\lp \frac{C}{\gamma^* \varepsilon} \rp^{\tfrac{1}{1-a}}  \rbr
\end{align}
Now, with any $\varepsilon, \gamma^*$ fixed, as we take $\alpha \downarrow 0$~(and hence $A \uparrow \infty$), we get
\begin{align}
\limsup_{ \alpha \downarrow 0} \frac{\EE[\tau]}{\log(1/\alpha)/\gamma^*} \leq \frac{1}{1-3\varepsilon}.     
\end{align}
Since the choice of $\varepsilon \in (0,1/3)$ was arbitrary, we can take an infimum over all such $\varepsilon$, which leads to the required statement 
\begin{align}
    \limsup_{\alpha \downarrow 0}  \frac{\EE[\tau]}{\log(1/\alpha)/\gamma^*} \leq 1. 
\end{align}
This concludes the proof. \hfill \qedsymbol

\section{Proof of~\Cref{prop:example}~(H\"older Densities)}
\label{proof:example}

Throughout the proof, we will use $\theta \in \Theta_0$ to  denote the (\holder continuous) density of any null distribution $P_\theta$, and $p$ to denote the density of any $P \in \calP_1$. We will verify each property of~\Cref{assump:general} one by one. 

\begin{lemma}
    \label{lemma:weak-compactness-Holder} The class $ \calP_0 = \{P_\theta: \theta \in \Theta_0\}$ is a weakly compact subset of $\calP(\calX)$. 
\end{lemma}
\begin{proof}
    Recall that $\calX = [0,1]^d$, and $\Theta_0$ consists of $(L,\beta)$ \holder continuous functions, with $c_0 \leq \theta(x) \leq C_0$ for all $x \in \calX$. Viewing $\Theta_0$ as a subset of $\calC(\calX, \R)$ endowed with the sup norm $\|\cdot\|_\infty$, observe that  
    \begin{itemize}
        \item $\|\theta\|_\infty = \sup_{x \in \calX} |\theta| \leq C_0$, which implies that $\Theta_0$ is a bounded subset of $(\calC(\calX, \R), \|\cdot\|_\infty)$. 
        \item If $\theta$ denotes the limit of  a sequence $\{\theta_i: i \geq 1\} \subset \Theta_0$ in $\unorm$, then $\theta \in \Theta_0$. It easily follows that $|\theta(x) - \theta(x')| \leq |\theta(x)-\theta_i(x)| + |\theta_i(x)-\theta_i(x')| + |\theta_i(x')-\theta(x')| \leq L\|x-x'\|^\beta + |\theta(x)-\theta_i(x)| + |\theta(x')-\theta_i(x')|$. Since $\|\theta_i - \theta\|_\infty \to 0$, the limit $\theta$ is also $(L,\beta)$ \holder continuous. A similar argument shows that $\theta$ must also be a density with $c_0 \leq \theta(x) \leq C_0$. In other words, $\Theta_0$ is a closed subset of $(\calC(\calX, \R), \|\cdot\|_\infty)$. 
        \item Finally, note that $\Theta_0$ is also equicontinuous. That is, for every $\epsilon>0$, there exists a $\delta = (\epsilon/L)^{1/\beta}$, such that for all $x, x' \in \calX$ with $\|x-x'\| \leq \delta$, we have $\sup_{\theta \in \Theta_0}|\theta(x)-\theta(x')|\leq \epsilon$. 
    \end{itemize}
    Hence, the conditions for applying the Arzel\'a-Ascoli theorem~(recalled in~Fact~\ref{fact:arzela-ascoli}) are satisfied and $\Theta_0$ is a compact subset of $(\calC(\calX, \R), \|\cdot\|_\infty)$. 

    To conclude the proof, let $\varphi: \calX \to \R$ denote any element of $(\calC(\calX), \R), \|\cdot\|_\infty)$, and consider any sequence $\process{\theta_i} = \{\theta_i \in \Theta_0: i \geq 1\}$ converging to $\theta \in \Theta_0$ in $\unorm$. Then, observe that 
    \begin{align}
        \lv \int_{\calX} \varphi dP_{\theta} - \int_{\calX} \varphi dP_{\theta_i} \rv = \lv \int_{\calX} \varphi(x)  \theta_i(x)dx - \int_{\calX} \varphi(x) \theta(x) dx \rv \leq \|\varphi\|_\infty \|\theta - \theta_i \|_\infty \int_{\calX} dx \;\;\stackrel{i \to \infty}{\longrightarrow}\;\; 0. 
    \end{align}
    This implies that $\theta_i \to \theta$ implies $P_{\theta_i} \Rightarrow P_{\theta}$; or, that the map $\theta \mapsto P_\theta$ is continuous (w.r.t. the weak topology on $\calP(\calX)$). 
    Therefore, $\calP_0 = \{P_\theta: \theta \in \Theta_0\}$ is the continuous image of a compact set, hence compact. 
\end{proof}

Now, suppose that the true distribution is $P \in \calP_1$. Then, by the construction of $\calP_1$, there exists a $\theta \in \Theta_0$, such that $\ell_\theta = dP/dP_{\theta} \in \lp \calC(\calX, \R),\unorm \rp$.
Since $\theta$ is continuous and bounded away from $0$, it follows that $P$ admits a continuous and bounded density $p(x) = \ell_\theta(x) \theta(x)$. By the weak compactness of $\calP_0$, let $P_{\theta^*}$ denote the reverse information projection of $P$ on $\calP_0$; that is, $P_{\theta^*} = \argmin_{P_{\theta'} \in \calP_0} \dkl(P \parallel P_{\theta'})$, and let $\theta^* \in \Theta_0$ denote its density with respect to $\mu$. Then, we can observe that 
\begin{align}
    \ell^*(x) \equiv \ell_{\theta^*}(x) \coloneqq \frac{dP}{dP_{\theta^*}}(x)= \frac{p(x)}{\theta^*(x)}, \quad \implies \quad \|\ell^*\|_\infty  \leq \frac{C_0}{c_0} \|\ell_\theta\|_\infty  < \infty, 
\end{align}
and $\ell^* > R_0 \frac{c_0}{C_0} > 0$ by the assumption on $\calP_1$ in~\Cref{def:nonparametric-example}. 
This implies that the optimal log-likelihood $f^* = \log \ell^*$ is also an element of $\calC(\calX, \R)$ with its $\sup$ norm bounded by $\log C_0 \|\ell_\theta\|/c_0$. 

Now, we know that the RKHS of a \emph{universal kernel} $\kappa$ is dense in $\lp \calC(\calX, \R), \unorm \rp$, which means that for an arbitrary $\epsilon>0$, there exists an $h_\epsilon \in \calH_{\kappa}$ such that $\|h_\epsilon - f^*\|_\infty \leq \epsilon$. Let $R_\epsilon $ denote the RKHS norm of $h_\epsilon$, and let $j(\epsilon) = \max \{0, \lceil \log_2 R_\epsilon \rceil \}$ denote the index of the smallest $\calF_j$ that contains $h_\epsilon$. Additionally, we have 
\begin{align}
    &\lv \EE_P[h_\epsilon(X)] - \EE_P[f^*(X)] \rv \leq \|h_\epsilon - f^*\|_\infty \leq \epsilon, \\
    \text{and} \quad & \psi_0(h_\epsilon) = \sup_{\theta \in \Theta_0} \log \EE_{\theta}[e^{h_\epsilon(X)}] \leq \sup_{\theta \in \Theta_0} \log \lp \EE_{\theta}[e^{f^*(X)}] e^{\epsilon} \rp \leq \psi_0(f^*) + \epsilon. 
\end{align}
Together, these two statements imply that $\EE_P[h_\epsilon(X)] - \psi_0(h_\epsilon) \geq \EE_P[f^*(X)] - \psi_0(f^*)  - 2\epsilon = \gamma^*(P, \calP_0) - 2\epsilon$.  In other words, for any accuracy level $\epsilon>0$, we can select a $j(\epsilon) < \infty$, such that we have 
\begin{align}
    \sup_{h \in \calF_{j(\epsilon)}} \EE_P[h(X)] - \psi_0(h) \geq \gamma^*(P, \calP_0) - 2 \epsilon. \label{eq:approx-DV-example} 
\end{align}
This establishes an approximate version of~\textbf{(A2)} in~\Cref{assump:general}, that as we show later, is sufficient. 

Because, the kernel $\kappa$ is bounded, it also means that every element of $\calF_{j(\epsilon)}$  satisfies 
\begin{align}
    \|h\|_\infty = \sup_{x \in \calX} |h(x)| =  \sup_{x \in \calX} |\langle h, \kappa(x, \cdot) \rangle_{\kappa}| \leq \sup_{x \in \calX} \sqrt{\kappa(x, x)} \|h\|_{\kappa} \leq 2^{j(\epsilon)}. 
\end{align}
As a result, the uniform moment requirement~\textbf{(A3)}, and the uniform MGF condition~\textbf{(A4)} are satisfied trivially. Finally, the uniform convergence condition~\textbf{(A5)} also holds with $r_n = \calO(2^{j(\epsilon)}\sqrt{\log n/n})$ using standard covering number bounds for RKHSs. 

Finally, fix an $\alpha$, and observe that 
\begin{align}
    W_n = \sum_{j \geq 0} c_j W_n^j \geq c_{j(\epsilon)} W_n^{j(\epsilon)}, \quad \implies \quad \tau_\alpha \leq \inf \lbr n \geq 1: W_n^{j(\epsilon)} \geq \frac{1}{c_{j(\epsilon)} \alpha} \rbr \eqcolon \tau_{\alpha, \epsilon}. 
\end{align}
Observe that $\tau_{\alpha, \epsilon}$ is based only on an $e$-processes associated with the function class $\calF_{j(\epsilon)}$, which, as we have shown above, satisfies the conditions \textbf{(A3)}-\textbf{(A5)} of~\Cref{assump:general}. Thus, an application of~\Cref{theorem:general-DV-e-process} implies that 
\begin{align}
     \EE_P\lb \tau_{\alpha, \epsilon} \rb  = \frac{\log(1/\alpha c_{j(\epsilon)})}{\gamma^*(P, \calP_0) - 2 \epsilon} \lp 1 + o(1) \rp, 
\end{align}
where $o(1)$ indicates a term that goes to zero with all other quantities~($P$, $\epsilon$, $j(\epsilon)$ etc.) fixed. Thus, on taking $\alpha \downarrow 0$, we get 
\begin{align}
 \limsup_{\alpha \downarrow 0} \frac{\EE_P\lb \tau_\alpha \rb}{\log(1/\alpha)/\lp \gamma^*(P, \calP_0) - 2 \epsilon\rp} \leq    \limsup_{\alpha \downarrow 0} \frac{\EE_P\lb \tau_{\alpha, \epsilon} \rb}{\log(1/\alpha)/\lp \gamma^*(P, \calP_0) - 2 \epsilon\rp}   \leq 1. 
\end{align}
Since $\epsilon>0$ was arbitrary, we may take it to $0$ without violating the bound. This concludes the proof.

\section{Auxiliary Lemmas}    
\begin{lemma}
    \label{lemma:uniform-erm}
    Under the conditions of~\Cref{assump:general}, the event $G_{n,1}$ introduced in~\eqref{eq:Gn1-Gn2-def} has probability at least $1-1/n^{2+\delta/3}$ for each $n \geq 2$. 
\end{lemma}

\begin{proof}
    This result follows from an application of part \textbf{(A5)} of~\Cref{assump:general}, and the definition of $f_n$. In particular, recall that $f_n$ is defined as the ERM solution  for all $n \geq 2$: 
    \begin{align}
        f_n \in \argmax_{f \in \calF} \frac{1}{n-1} \sum_{i=1}^{n-1} f(X_i) - \psi_0(f). 
    \end{align}
    By~\textbf{(A5)}, we know that for all $f \in \calF$ and $n \geq 2$, we have the following with probability at least $1-1/n^{2 + \delta/3}$: 
    \begin{align}
        \left\lvert \EE_P[f(X)] - \frac{1}{n-1} \sum_{i=1}^n f(X_i) \right\rvert \leq r_n \; \implies \; \left\lvert \lp \EE_P[f(X)] - \psi_0(f)\rp - \lp \frac{1}{n-1} \sum_{i=1}^n f(X_i) - \psi_0(f) \rp \right\rvert \leq r_n. \label{eq:proof-erm-1}
    \end{align}
    As a result, we have the following, with $\Phat_{n-1}$ denoting the empirical distribution of $X^{n-1}$: 
    \begin{align}
        \EE_{\Phat_{n-1}}[f_n(X)] - \psi_0(f_n) \;\geq\; \EE_{\Phat_{n-1}}[f^*(X)] - \psi_0(f^*) \; \geq \; \EE_P[f^*(X)] - \psi_0(f^*) - r_n = \gamma^*(1 - \varepsilon)  - r_n, 
    \end{align}
    where the first inequality is due to $f_n$ being the ERM function, and the second inequality is due to~\eqref{eq:proof-erm-1}. Another application of~\eqref{eq:proof-erm-1} then gives us 
    \begin{align}
        \gamma^*(1-\varepsilon) - r_n \; \leq \; \EE_{\Phat_{n-1}}[f_n(X)] - \psi_0(f_n) \;\leq\; \EE_P[f_n(X)] - \psi_0(f_n) + r_n, 
    \end{align}
    or equivalently, $\EE_P[f_n(X)] \geq \gamma^*(1-\varepsilon) - 2 r_n$ under the $1-1/n^{2 + \delta/3}$ event in~\textbf{(A5)}. This concludes the proof, noting that $\EE_P[V_n \mid \calM_{n-1}] = \EE_P[f_n(X_n) - \psi_0(f_n) \mid \calM_{n-1}] = \EE_P[f_n(X)] - \psi_0(f_n)$, as $X_n \perp \calM_{n-1}$. 
\end{proof}

\begin{lemma}
    \label{lemma:martingale-concentration}
    Under the conditions of~\Cref{assump:general}, the event $G_{n,2}$ defined in~\eqref{eq:Gn1-Gn2-def} for $n \geq 2$ occurs with probability at least $1- n^{-2 - \delta/3}$. 
\end{lemma}
\begin{proof}
    We will begin by applying~\Cref{fact:pinelis-1994} with $Y_i = \bar{V}_i = V_i - \EE[V_i \mid \calM_{i-1}]$, and for $p = 4 + \delta$, we have 
    \begin{align}
        |\Vbar_i|^p & \leq 2^{p-1} \lp  |V_i|^p + |\EE[V_i \mid \calM_{i-1}]|^p \rp  \leq 2^{p-1} \lp  |V_i|^p + \EE[|V_i |^p \mid \calM_{i-1}]\rp, 
    \end{align}
    where we use $(a+b)^p = 2^{p} \lp \tfrac{a+b}{2} \rp^p \leq 2^{p}\lp \tfrac{a^p}{2} + \tfrac{b^p}{2} \rp = 2^{p-1}(a^p + b^p)$ for constants $a, b >0$ and $p>1$ in the first inequality, while the second inequality is by conditional Jensen's. 
    This means that 
    \begin{align}
        \EE\lb \lv \Vbar_i \rv^p \rb & \leq 2^p \EE\lb |V_i|^p  \mid \calM_{i-1} \rb = 2^p \EE\lb |f_i(X_i) - \psi_0(f_i)|^p \mid \calM_{i-1} \rb \leq 2^{2p-1} \EE\lb |f_i(X_i)|^p + |\psi_0(f_i)|^p  \mid \calM_{i-1}\rb. 
    \end{align}
    By condition~\textbf{(A4)}, we know that $\sup_{f\in \calF} |\psi_0(f)| \leq B_0$, and so $\EE[|\psi_0(f_i)|^p \mid \calM_{i-1}]$ can be bounded by $ B_0^p$. For the remaining term, note that 
    \begin{align}
        \EE[|f_i(X_i)|^p \mid \calM_{i-1}] & \leq \sup_{f \in \calF}\EE \lb |f(X_i)|^p \mid \calM_{i-1} \rb   \leq B_P, 
    \end{align}
    by \textbf{(A3)} condition of~\Cref{assump:general}. Together, these inequalities imply the bound 
    \begin{align}
        \EE[|\Vbar_i|^p \mid \calM_{i-1}] \leq 2^{2p-1}\lp B_P + B_0^p \rp \eqcolon K   < \infty, \label{eq:concentration-moment-bound-1}
    \end{align}
    almost surely. 
    Now, by~\Cref{fact:pinelis-1994}, we have the following for constants $C_{1,p}, C_{2,p}<\infty$: 
    \begin{align}
        \EE[| \max_{1 \leq k \leq n} S_k |^p]^{1/p} & \leq C_{1,p} \lp \EE\lb \max_{1 \leq i \leq n} |\Vbar_i|^p \rb \rp^{1/p} + C_{2,p} \lp \EE\lb \lp\sum_{i=1}^n \EE\lb \Vbar_i^2 \mid \calM_{i-1} \rb  \rp^{p/2} \rb \rp^{1/p} \\
        & \leq C_{1,p} \lp \EE\lb \sum_{i=1}^n |\Vbar_i|^p \rb \rp^{1/p} + C_{2,p} \lp \EE\lb \lp\sum_{i=1}^n \EE\lb \Vbar_i^2 \mid \calM_{i-1} \rb  \rp^{p/2} \rb \rp^{1/p}. 
    \end{align}
    Hence, there exists some $C< \infty$, such that 
    \begin{align}
        \EE\bigg[| \max_{1 \leq k \leq n} S_k |^p \bigg]
        & \leq C \lp \EE\lb \sum_{i=1}^n \EE \lb |\Vbar_i|^p  \mid \calM_{i-1} \rb \rb  + \EE\lb \lp\sum_{i=1}^n \EE\lb \Vbar_i^2 \mid \calM_{i-1} \rb  \rp^{p/2} \rb  \rp. 
    \end{align}
    The first term in the RHS above can be upper bounded by $C \times n K$ due to~\eqref{eq:concentration-moment-bound-1}.  For the second term, observe that 
    \begin{align}
        \EE[|\Vbar_i|^2 \mid \calM_{i-1}] & = \EE \lb\lp |\Vbar_i|^p\rp^{2/p} \mid \calM_{i-1} \rb \stackrel{\text{Jensen's}}{\leq}  \lp\EE \lb |\Vbar_i|^p \mid \calM_{i-1} \rb \rp^{2/p} \leq K^{2/p}, 
    \end{align}
    where the last inequality again uses~\eqref{eq:concentration-moment-bound-1}. Thus, we obtain the following 
    \begin{align}
        \EE\bigg[| \max_{1 \leq k \leq n} S_k |^p \bigg] & \leq C K n + C K n^{p/2} \leq
        C'_p n^{p/2}, 
    \end{align}
    for some constant $C'_p$.

    Thus, for any $t_n >0$, an application of Markov's inequality gives us 
    \begin{align}
        \PP\lp \max_{1 \leq k \leq n} |S_k| \geq (C_p')^{1/p}t_n \rp \leq  \EE\lb \max_{1 \leq k \leq n} |S_k|^p\rb t_n^{-p} (C_p')^{-1}\leq  n^{p/2} t_n^{-p}. 
    \end{align}
    Selecting $t_n = n^a$ with $a = \tfrac {2 + \delta/3}{p} + \tfrac 12$, we get $\max_{1 \leq k \leq n} |S_k| \leq Cn^a$~(where we use $C = (C_p')^{1/p}$) with probability at least $1 -  n^{-2 - \tfrac{\delta}{3}}$. This completes the proof
\end{proof}

    \begin{lemma}
        \label{lemma:fixed-point-1} For some positive constants $K$ and $L$, let $y^*$ denote the largest solution of the equation 
        \begin{align}
            y \;=\; K + L \log y.
        \end{align}
        Then, we have the following implication: 
        \begin{align}
             \frac{K}{L} > \max \lbr 1 - \log L, \;  \frac{\log K + 3}{2}  \rbr \quad \implies \quad   y^* \leq K + L \log K  + 2L. 
        \end{align}
        In other words, as $K \uparrow \infty$ with $L$ fixed, we have $y^* = \lp K + L \log K \rp (1 + o(1))$, with $o(1)$ indicating a term that vanishes to $0$ as $K \uparrow \infty$. 
    \end{lemma}

    \begin{proof}
        The function $f(y) = y - K  - L \log y$ is convex on the domain $(0, \infty)$, and it satisfies $\lim_{y \downarrow 0} f(y) = \lim_{y \uparrow \infty} f(y) = + \infty$. Hence it achieves its minimum value at $y$ such that $f'(y) = 1 - L/y = 0$ or $y_{\min} = L$. The function value at $y_{\min}$ is $L - K - L \log L < 0$ by the assumption on $K/L$. Thus the equation $f(y)=0$ has two solutions, say $y_0 < y_1$, and our goal is to get an upper bound on $y^* = y_1$. To do this, note that it is sufficient to find a $y > L$ such that $f(y) > 0$. Let us try a value $y = K + L \log K + C$ for some constant $C>0$ that we will specify later, and observe that $f(y) = L \log K + C - L \log (K + L \log K + C) = C - L \log \lp 1 + \frac{L \log K + C}{K} \rp \geq C - L\lp \frac{L \log K + C}{K} \rp$, where we used the fact that $\log(1+x) \leq x$ for $x \geq 0$. Now, suppose $2L \leq C \leq 3L$, which gives us the bound $f(y) \geq C - L^2\lp\frac{\log K + 3}{K} \rp$ which is greater than $0$ by the assumption on $(L, K)$. Hence, we can conclude that $y^* \leq L \log K + K + 2L$.

    \end{proof}

\end{appendix}

\end{document}